\documentclass[12pt,reqno]{amsart}
\usepackage[margin=0.95in]{geometry}
\usepackage{amsmath,amssymb,amsthm,graphicx,amsxtra, setspace}
\usepackage[utf8]{inputenc}
\usepackage{mathrsfs}
\usepackage{hyperref}
\usepackage{upgreek}
\usepackage{mathtools}
\usepackage[mathcal]{euscript}
\usepackage{xcolor}
\allowdisplaybreaks

\usepackage[pagewise]{lineno}

\usepackage{graphicx,eurosym}
\usepackage{hyperref}
\usepackage{mathtools}

\usepackage[cyr]{aeguill}

\colorlet{darkblue}{blue!50!black}

\hypersetup{
	colorlinks,%
	citecolor=blue,%
	filecolor=red,%
	linkcolor=darkblue,%
	urlcolor=blue,%
	pdfnewwindow=true,%
	pdfstartview={FitH}
}

\usepackage{graphicx,amscd,mathrsfs,wrapfig,mathrsfs,lipsum}
\usepackage{tikz}
\usepackage{multicol}
\usepackage{caption}
\usetikzlibrary{arrows}

\colorlet{darkblue}{blue!50!black}

\DeclareUnicodeCharacter{00A0}{ }
\makeatletter
\let\ltxxlabel\ltx@label
\makeatother

\newtheorem{theorem}{Theorem}[section]
\newtheorem{lemma}[theorem]{Lemma}
\newtheorem{proposition}[theorem]{Proposition}

\newtheorem{definition}[theorem]{Definition}

\newtheorem{remark}[theorem]{Remark}
\newtheorem{hypothesis}[theorem]{Hypothesis}

\newtheorem{condition}[theorem]{Condition}

\let\emptyset\varnothing

\let\originalleft\left
\let\originalright\right
\renewcommand{\left}{\mathopen{}\mathclose\bgroup\originalleft}
\renewcommand{\right}{\aftergroup\egroup\originalright}

\theoremstyle{definition}

\def\vi{\widetilde}

\def\d{\mathrm{d}}

\def\I{\mathrm{I}}

\def\m{\mathrm{m}}
\def\W{\mathrm{W}}
\def\R{\mathbb{R}}
\def\E{\mathbb{E}}
\def\Q{\mathrm{Q}}
\def\H{\mathrm{H}}

\def\e{\varepsilon}
\def\2{\mathcal{E}}
\def\L{\mathrm{L}}

\def\C{\mathrm{C}}

\def\F{\mathrm{F}}
\def\P{\mathbb{P}}
\def\N{\mathbb{N}}

\def\G{\mathrm{G}}

\def\Z{\mathrm{Z}}
\def\PP{\mathscr{P}}

\def\UU{\mathfrak{U}}
\def\SS{\mathbb{S}}

\def\p{\mathfrak{p}}
\def\c{\mathfrak{c}}
\def\EE{\mathcal{E}}

\newcommand{\Addresses}{{
		\footnote{
			\noindent 	\textsuperscript{1}Theoretical Statistics and Mathematics Unit, Indian Statistical Institute Banglore Centre-ISI Banglore, 8th Mile Mysore Road, Bangalore, Karnataka 560059, India\\
			\textsuperscript{2}Department of Mathematics,
			Montanuniversitaet Leoben,
			Austria\\
			\textsuperscript{2,3}Department of Mathematics, Indian Institute of Technology Roorkee-IIT Roorkee,
			Haridwar Highway, Roorkee, Uttarakhand 247667, India.
			\par\nopagebreak
			\noindent  \textit{e-mail:} \texttt{Manil T. Mohan: maniltmohan@ma.iitr.ac.in, maniltmohan@gmail.com.}
			
			\textit{e-mail:} \texttt{Ankit Kumar: akumar14@mt.iitr.ac.in,ankit.kumar@unileoben.ac.at.}
			
			\textit{e-mail:} \texttt{Vivek Kumar: vivekkumar\_ra@isibang.ac.in, vivekmsc118@gmail.com.}
			
			\noindent \textsuperscript{*}Corresponding author.
			
			\textit{Keywords:} Stochastic Burgers-Huxley equations, mild solution, central limit theorem, moderate deviations.

			Mathematics Subject Classification (2020): Primary 60H15, 60F10; Secondary 37L55, 60H20.

}}}

\begin{document}	
	
	\title[ Asymptotic behavior of the SGBH equation]{ Central limit theorem  and moderate deviation principle  for the stochastic generalized  Burgers-Huxley equation with multiplicative noise
		\Addresses}

	\author[V. Kumar, A. Kumar and M. T. Mohan]
	{Vivek Kumar\textsuperscript{1}, Ankit Kumar\textsuperscript{2} and Manil T. Mohan\textsuperscript{3*}}

	\maketitle

	\begin{abstract}
		In this work, we investigate the Central Limit Theorem (CLT) and Moderate Deviation Principle (MDP) for the stochastic generalized  Burgers-Huxley (SGBH) equation  with multiplicative Gaussian  noise. The SGBH equation is a diffusion-convection-reaction type equation which consists a nonlinearity of polynomial order, and we take into account of an infinite-dimensional noise having a coefficient that has linear growth. We first  prove the CLT which allows us to establish the convergence of the distribution of the solution to a re-scaled SGBH equation to a desired distribution function. Furthermore, we extend our asymptotic analysis by  investigating the MDP for the SGBH equation.  Using the weak convergence method, we establish the MDP and derive the corresponding rate function.
	\end{abstract}

	\section{Introduction}\label{Intro}\setcounter{equation}{0}
	The study of deviations in stochastic processes plays a crucial role in understanding the behavior of random phenomena. Deviation analysis allows us to quantify the likelihood of observing outcomes that significantly deviate from the expected or typical behavior. Large Deviation Principle (LDP), Central Limit Theorem (CLT), and Moderate Deviation Principle (MDP) are three important frameworks that provide insights into different scales of deviations in stochastic processes, see \cite{BBM,QCHG, DV85,MIFADW,VJE,HLW,MTMCLT,MTM2020,V66,V84,XZ,ZZG}, etc., and references therein. The theory of  large deviation is a mathematical tool that helps us to understand and measure rare events or extreme changes in random processes. Traditional probability theory focuses on how random quantities  behave around their average values, but large deviation theory looks at the less likely outcomes. It gives us a better understanding of the probabilities of unusual events. CLT relates to the behavior of sums or averages of many independent random variables, focusing on a typical behavior with small deviations. MDP bridges the gap between LDP and CLT by analyzing moderate deviations from the mean behavior of stochastic processes. It provides a mathematical framework to understand the probabilities associated with events that are not extremely rare but still deviate significantly from the expected behavior.

	Mathematically, let us examine a sequence of stochastic processes denoted as $\{u_\e(t): t>0\}_{\e\geq 0}$, and let  $u_0(\cdot)$ represents the limit of $u_\e(\cdot)$, as $\e\to0$. Our focus is on understanding the deviation of $u_\e(\cdot)$ from the deterministic behavior of $u_0(\cdot)$, as $\e\to0$. That means, we  study the asymptotic behavior of  the trajectory
	\[\mathrm{Z}_\e(t):=\frac{u_\e(t)-u_0(t)}{\sqrt{\e}\lambda(\e)},\]
	where $\lambda(\e)$ is some deviation scale. In particular, it influences the asymptotic behavior of $u_\e(\cdot)$ in the following way:
	\begin{enumerate}
		\item  for $\lambda(\e)=\frac{1}{\sqrt{\e}},$ we get the LDP;
		\item for $\lambda(\e)=1$, it gives the CLT;
		\item  for $\lambda(\e)\rightarrow \infty$ and $\sqrt{\e}\lambda(\e)\rightarrow 0$, as $\e \rightarrow 0$ (for example $\lambda(\e)=\e^{-\theta}$ for some $0<\theta<\frac{1}{2}$), it provides the MDP.
	\end{enumerate}
	In this work, we are interested to establish the above mentioned stochastic deviation properties for the solutions to the following generalized version of Burgers-Huxley equation driven by a random forcing indexed by $\e\in(0,1]$:
	\begin{align}\label{1.1}
		\frac{\partial u_\e(t,x)}{\partial t}&=\nu\frac{\partial^2 u_\e(t,x)}{\partial x^2}- \alpha u_\e^\delta(t,x) \frac{\partial u_\e(t,x)}{\partial x}
		+\beta u_\e(t,x)(1-u_\e^\delta(t,x))(u_\e^\delta(t,x)-\gamma)\nonumber\\
		&\quad\qquad+ \sqrt{\e}g(t, x, u_\e(t,x))\mathcal{W}(t,x),~~~~~t>0,~ x\in (0,1),
	\end{align}where 
	\begin{itemize}
		\item   $u^{\e}(t)=u^{\e}(t,x)$ signifies the concentrations or density of a specific thing of interest, such as the population density of a chemical species or a biological environment;
		\item $\delta\in\N$ is a parameter which represents the degree of nonlinearity within the equation, subsequently controlling the magnitude of the response or feedback term in the equation;
		\item $\alpha>0$ denotes the coefficient associated with the nonlinear advection term, capturing the energy of a transport phenomena or  a nonlinear wave propagation. It is known as the advection coefficient;
		\item $\nu$ denotes the coefficient of diffusion term which helps us to understand  the spatial dispersion or spreading of the quantity $u$. It is known as the viscosity coefficient;
		\item $\beta>0$ denotes the parameter associated to the reaction or feedback term, controlling the nonlinear reaction's or growth process's intensity;
		\item $\gamma\in(0,1)$ is another parameter associated with the reaction or feedback term, impacting the behavior of the system at different concentration levels;
		\item $\e\in(0,1]$ is the intensity of the noise.
	\end{itemize}
	The problem \eqref{1.1} is associated  with the Dirichlet boundary condition and the deterministic  initial data 
	\begin{align}\label{1.2}
		u_\e(t,0)=u_\e(t,1)=0,\ \ t\geq0, \ \ \text{ and } \ \ u_\e(0,x)=u^0(x),
	\end{align}respectively. The noise coefficient $g(\cdot,\cdot,\cdot)$ has linear growth and it satisfies a Lipschitz condition  in the third variable (see Hypothesis \ref{hyp1} below).

	The problem \eqref{1.1} without random forcing  (that is, $\e=0$), with $\beta=0$ and $\delta=1$, is the \textsl{classical viscous Burgers equation} which has several real life applications, see \cite{BJM}. Further, when $\e=0, \beta\neq 0$ and $\delta$ is set to be $1$, then the reduced equation is known as \textsl{Burgers-Huxley equation}, for instance see \cite{VJE,MTMAK}, etc., and references therein. This equation finds widespread application in the analysis of various physical and biological systems, such as fluid dynamics, neural networks, and pattern formation. For $\beta=0$ and $\alpha=1$, the equation \eqref{1.1} forms the following generalized stochastic Burgers equation:
	\begin{align}\label{Burgers}
		\frac{\partial u_\e(t,x)}{\partial t}=\nu\frac{\partial^2 u_\e(t,x)}{\partial x^2}-\alpha u_\e^\delta(t,x)\frac{\partial u_\e(t,x)}{\partial x}+\sqrt{\e}g(t,x,u_\e(t,x))\mathcal{W}(t,x),
	\end{align} 
	which is widely known and extensively studied as a simplified model for understanding fluid turbulence. It has received significant attention in the literature, with numerous studies and publications dedicated to its exploration. Notable references include \cite{DDR, GDDG, IG, IGDN, VMG} and references therein, which provide further insights and analysis on this topic.   If we choose  $\alpha=0$, then the equation \eqref{1.1} reduces to the following equation:
	\begin{align}\label{Huxley}\nonumber
		\frac{\partial u_\e(t,x)}{\partial t}&=\nu\frac{\partial^2 u_\e(t,x)}{\partial x^2}+\beta u_\e(t,x)(t-u_\e^\delta(t,x))(u_\e^\delta(t,x)-\gamma)\\&\quad +\sqrt{\e}g(t,x,u_\e(t,x))\mathcal{W}(t,x), \ \ \text{ for } \ \ (t,x)\in(0,T)\times(0,1).
	\end{align}
	The equation \eqref{Huxley} is known as the \textsl{stochastic Huxley equation}. It serves as a mathematical representation for the transmission of nerve impulses in nerve fibers and the movement of liquid crystals in walls (for more details, see \cite{XW}). 
	The SGBH equation \eqref{1.1} covers both the models given by \eqref{Burgers} and \eqref{Huxley}. The deterministic generalized Burgers-Huxley (GBH) equation has been extensively studied in \cite{VJE, MTMAK,SLRRSMTM}, and  references therein, along with numerical investigations.  Regarding the global solvability results of its stochastic counterpart, one can find comprehensive information in the works \cite{HGXLJM1,HGXLJM2,AKMTM3, AKMTM2, MTMSBH, MTMSGBH}.

	\subsection{A short note on the noise $\mathcal{W}(\cdot,\cdot)$}\label{SN}
	Let us now characterize the noise $\mathcal{W}(\cdot,\cdot)$ 
	appearing in \eqref{1.1} on the probability space $\big(\Omega,\mathscr{F},\{\mathscr{F}_t\}_{t\geq0},\P\big)$ (see \cite{SVLBLR}). Consider a sequence  $\{\xi_j\}_{j\geq 1 }$ of independent standard Gaussian random variables on this basis (see \cite[Lemma II.2.3]{NVK}). If $\{\m_j(t)\}_{j\in\N}$ is an  orthonormal basis in $\mathrm{L}^2([0,T]), $ then
	
	\[\beta(t):=\sum_{j\in\N} \bigg(\int_{0}^{t}\m_j(s)ds\bigg) \xi_j\]
	is a standard Brownian motion with covariance $\E[\beta(t)\beta(s)]=(t\wedge s)=\min\{t,s\}$. 
	Consider a sequence $\{\beta_j\}_{j\in\N}$ of independent standard Brownian motions and  an orthonormal basis $\{\varphi_j\}_{j\in \N}$ in the space $\L^2([0,1])$. Let us define
	\[h_j(x):=\int_{0}^{x} \varphi_j(y)dy.\] 
	Then the process
	$$\W(t,x):=\sum_{j\in\N}h_j(x)\beta_j(t)$$
	is Gaussian  with the mean zero and covariance $\E[\W(t,x)\W(s,y)]=(t\wedge s)(x\wedge y).$ The process $\W$ is known as the \textsl{Brownian sheet.} The formal term-by-term differentiation of the series in $\beta$ suggests a representation
	\begin{align*}
		\dot{\beta}(t)=\sum_{k\in\N}\m_k(t)\xi_k,
	\end{align*}
	so that 
	\begin{align}
		\frac{\partial^2\W}{\partial t\partial x}(t,x)=\dot{\W}(t,x)=\sum_{j\in\N}\varphi_j(x)\dot{\beta}_j(t), 
	\end{align}
	and we call the process $\dot{\W},$ the \textsl{(Gaussian) space-time white noise.} Even though the series  diverges, it defines a random generalized function on $\L^2([0,T]\times[0,1])$:
	\begin{align*}
		\dot{\W}(f)=\sum_{j\in\N}\int_0^T\bigg(\int_0^1f(t,x)\varphi_j(x)\d x\bigg)\d \beta_j(t)=\int_0^T\int_0^1f(t,x)\W(\d t,\d x). 
	\end{align*}
	
	To construct a noise, that is, white in time and colored in space, take a sequence of non-negative numbers $\{q_j\}_{j\in\N}$ and define 
	\begin{align}\label{15}
		\frac{\partial^2\W^{\Q}}{\partial t\partial x}(t,x)=\dot{\W}^{\Q}(t,x)=\mathcal{W}(t,x)=\sum_{j\in\N}q_j\varphi_j(x)\dot{\beta}_j(t),
	\end{align}	
	and we call the process $\dot{\W}^\Q,$ the \textsl{(Gaussian) white in time and colored in space noise.} Whenever we consider $\dot{\W}^\Q$ in the work, we assume that \begin{align}\label{181}\sum\limits_{j\in\N}q_j^2<\infty.\end{align} 
	Let  the family  $\{\varphi_j\}_{j\in\N}$ be the orthonormal basis of $\L^2([0,1])$ consisting of the eigenfunctions of the operator $-\frac{\partial^2}{\partial x^2}$ corresponding to the eigenvalues $\{\lambda_j\}_{j\in\N}$. In fact, for $j=1,2,\ldots$
	\begin{equation*}
		\lambda_j=j^2 \pi^2, \ 	\varphi_j(x)=\sqrt{2}\sin(j\pi x).\
	\end{equation*} Then an example of a family $\{q_j\}_{j\in\N}$ satisfying \eqref{181}  is $\{\lambda_j^{-\eta}\}_{j\in\N}$ for $\eta>\frac{1}{4}$, since $\sum\limits_{j\in\N}\frac{1}{j^{4\eta}}<\infty$ provided $\eta>\frac{1}{4}$.

	\subsection{ Objectives, literature survey, novelties and methodology}
	In a recent work \cite{KKM}, the authors established a uniform large deviation principles (ULDP) for the SGBH equation \eqref{1.1}. The LDP and ULDP were obtained using the weak convergence approach. Building upon the findings of \cite{KKM}, this current study aims to derive the CLT and MDP for the system \eqref{1.1}-\eqref{1.2}. Thus, this work serves as a continuation and extension of the results presented in \cite{KKM}.
	
	There are extensive results are available in the literature for the CLT and MDP for stochastic partial differential equations (SPDEs). For notable contributions in this field, we refer the readers to \cite{BBM, HLW, MTMCLT,MTM2020, WZZ, XZ, ZZG}, etc., and references therein. In the context of our current investigation, we  focus primarily on the recent studies \cite{BBM, HLW,XZ} that are relevant to our specific research objectives.
	
	In the work \cite{BBM}, the authors  established the MDP and CLT for stochastic Burgers equation. Similarly, the studies conducted in \cite{HLW, XZ} developed the CLT and MDP for a class of semilinear SPDEs. It is worth noting that in all three of these works \cite{BBM, HLW, XZ}, the noise coefficient $g(t,x,r)$ was assumed to be bounded by a constant and exhibit Lipschitz continuity with respect to the third variable $r$. They have considered space-time white noise with such type of coefficient $g$ and having  atmost second-order nonlinear advection-type of function in their respective SPDE models.
	
	In contrast, our present study focuses on SPDEs that incorporate  advection-term with a polynomial order and an additional reaction term characterized by polynomial order nonlinearity. Further,  we consider a noise process that is white in time and colored in space, with the noise coefficient $g(t,x,r)$ characterized by linear growth and Lipschitz continuity with respect to the third variable $r$.    
	
	The presence of a polynomial type of nonlinearity and the infinite-dimensionality of noise with an unbounded noise coefficient are some significant challenges to achieve our aims. The polynomial type nonlinearites are handled by using Green's function estimates  (see Appendix \ref{SUR}). 
	We overcame the difficulty of  a proper  Burkholder-Davis-Gundy (BDG) inequality in $\L^p$-spaces by using the results from \cite{MCV,IY}. Recently, the authors in \cite{MCV,IY} obtained BDG inequality in UMD Banach spaces,  which is quite helpful in handling various maximal inequalities within our analysis. Precisely, the authors in \cite{MCV} showed that if $\mathrm{M}$ is an $\L^p$-bounded martingale, $1<p<\infty$, with $\mathrm{M}_0=0$, that takes values in a UMD Banach space $(\mathrm{X},\|\cdot\|_{\mathrm{X}})$ over a measure space $(S,\Sigma,\mu)$, then for all $t\geq 0$, 
	\begin{align}\label{18}
		\E\bigg[\sup_{s\in[0,t]}\|\mathrm{M}_s(\sigma)\|_{\mathrm{X}}^p\bigg]\leq C_{p,\mathrm{X}}\mathbb{E}\bigg[\|[\mathrm{M}(\sigma)]_t^{\frac{1}{2}}\|_{\mathrm{X}}^p\bigg],
	\end{align}
	where the quadratic variation $[\mathrm{M}(\sigma)]_t$ is considered pointwise in $\sigma\in S$. As $\L^p([0,1]),$ $1<p<\infty$ is a UMD Banach space, we can use the above BDG inequality in the sequel.

	The proof of the MDP is similar to the proof of the {LDP} and is adapted from the works \cite{ABPFD,ADPDVM,MIFADW}. In \cite{MIFADW}, the authors developed a variant of the ULDP known as the Freidlin-Wentzell uniform large deviation principle (FWULDP) in their honor. They demonstrated the FWULDP for various families of stochastic dynamical systems in infinite dimensions by employing specific variational representations. {Additionally, in the work \cite{ADPDVM}, the authors developed a weak convergence method that utilizes the uniform Laplace principle to prove the LDP, where the uniformity is considered over compact subsets of some appropriate Polish space}. They also provided a sufficient condition for obtaining the ULDP for the distribution of solutions to certain SPDEs, which requires examining essential qualitative properties such as existence, uniqueness, and tightness of controlled analogues of the original SPDEs. For our problem, we primarily followed the techniques developed in \cite{ADPDVM} to prove an MDP for our model. We point out here that the stopping time arguments and \cite[Lemma 4.1]{ICAM2} play a crucial role  in the analysis presented in this paper.

	\subsection{Organization} The structure of this work is as follows: In Section \ref{Sec2}, we provides the mathematical formulations of the problem along with our primary assumption as well as the well-posedness results for the system \eqref{1.1}-\eqref{1.2}. In Section \ref{Sec3}, we establish a CLT, while in Section \ref{Sec4}, our focus is on establishing the MD for the current problem. The article concludes by referencing Appendix \ref{SUR}, where we recall some relevant results from \cite{IG,IGCR}. 
	
	\section{Assumptions and solvability results}\label{Sec2}\setcounter{equation}{0}
	We first discuss the function spaces used in this work and then provide the assumptions on the noise coefficient $g(\cdot,\cdot,\cdot)$. Finally, we provide the existence and uniqueness of solutions of the problem \eqref{1.1}-\eqref{1.2}.  
	
	\subsection{Function spaces} Let us denote the space of all infinite times differentiable functions having compact support in $[0,1]$ by $\C_0^\infty([0,1])$. The Lebesgue spaces are denoted by $\L^p([0,1])$, for $p\geq1$, and the norm in the space $\L^p([0,1])$ is denoted by $\|\cdot\|_{\L^p}$. For $p=2$, the inner product in the space $\L^2([0,1])$ is represented by $(\cdot,\cdot)$. Let us represent the Hilbertian Sobolev spaces by $\H^k([0,1])$, for $k\in\N$ and $\H_0^1([0,1])$ represent the closure of $\C_0^\infty([0,1])$ in $\H^1$-norm.  Since we are working in a bounded domain, by the Poincar\'e inequality ($\lambda_1\|u\|_{\L^2}^2\leq \|\partial_xu\|_{\L^2}^2$), we infer that the norm $(\|\cdot\|_{\L^2}^2+\|\partial_x\cdot\|_{\L^2}^2)^\frac{1}{2}$ is equivalent to the seminorm $\|\partial_x\cdot\|_{\L^2}$ and hence $\|\partial_x\cdot\|_{\L^2}$ defines a norm in $\H_0^1([0,1])$. Moreover, we obtain the Gelfand triplet $\H_0^1([0,1])\hookrightarrow \L^2([0,1])\hookrightarrow\H^{-1}([0,1])$, where $\H^{-1}([0,1])$ stands for the dual of  $\H_0^1([0,1])$. For the bounded domain $[0,1]$, one has the compact embedding $\H_0^1([0,1])\hookrightarrow\L^2([0,1])$. The duality paring between $\H_0^1([0,1])$ and its dual $\H^{-1}([0,1])$ as well as between $\L^p([0,1])$ and its dual $\L^{\frac{p}{p-1}}([0,1])$ is denoted by $\langle\cdot,\cdot\rangle$. In one dimension, the embedding $\H^s([0,1])\hookrightarrow\L^p([0,1])$ is compact for any $s>\frac{1}{2}-\frac{1}{p}$, for $p>2$. 
	\subsection{Hypotheses  and mild solution} Let us introduce the linear growth and Lipschitz conditions on the noise coefficient.
	\begin{hypothesis}\label{hyp1}
		The function $g:[0,T]\times[0,1]\times\R\to\R$ is a measurable function, satisfying the following conditions: for all $t\in[0,T],\ x\in[0,1]$ and $r,s\in\mathbb{R}$,
		\begin{align}\label{2.1}
			|g(t,x,r)|\leq K(1+|r|),\ \ \text{ and } \ \ |g(t,x,r)-g(t,x,s)|\leq L|r-s|,			
		\end{align} where $K$ and $L$ are some positive constants.
	\end{hypothesis}

	The \textsl{fundamental solution} $G(t,x,y)$ of the heat equation in the interval $[0,1]$ with the Dirichlet boundary conditions is given by 
	\begin{align}\label{19}
		G(t,x,y)=\frac{1}{\sqrt{4\pi t}}\sum_{m=-\infty}^{\infty}\bigg[e^{-\frac{(y-x-2m)^2}{4t}}-e^{-\frac{(y+x-2m)^2}{4t}}\bigg],
	\end{align}
	for all $t\in(0,T]$ and $x,y\in[0,1]$. 
	
	Next, we provide the definition of \textsl{mild solution} in the sense of Walsh (see \cite{JBW}) to the problem \eqref{1.1}-\eqref{1.2}. 
	\begin{definition}[Mild solution]\label{def1}
		An $\L^p([0,1])$-valued and $\mathscr{F}_t$-adapted stochastic process $u_\e:[0,\infty)\times[0,1]\times\Omega\to\R$ with $\P$-a.s.~continuous trajectories on $t\in[0,T]$, is called a \emph{mild solution} to the problem \eqref{1.1}-\eqref{1.2}, if for any $T>0$, the following integral equation is satisfied: for all $t\in[0,T]$, $\P$-a.s.,
		\begin{align}\label{mild1}\nonumber
			u_\e(t,x)&=\int_0^1G(t,x,y)u^0(y)\d y+\frac{\alpha}{\delta+1}\int_0^t\int_0^1\frac{\partial G}{\partial y}(t-s,x,y)\p(u_\e(s,y))\d y\d s\\&\nonumber\quad +\beta\int_0^t\int_0^1G(t-s,x,y)\c(u_\e(s,y))\d y\d s\\&\quad+ \sqrt{\e}
			\int_0^t\int_0^1G(t-s,x,y)g(s,y,u_\e(s,y))\W^\Q(\d s,\d y), 
		\end{align} where $\p(u)=u^{\delta+1}, \ \c(u)=u(1-u^\delta)(u^\delta-\gamma)$ and $G(\cdot,\cdot,\cdot)$ is the fundamental solution of the heat equation in $[0,1]$ with the Dirichlet boundary conditions.
	\end{definition}
	Assume that the initial data $u^0\in\L^p([0,1])$, $p\geq 2$.  From \cite[Proposition 3.7]{IGCR}, we know that $u_\e$ is a mild solution to the system \eqref{1.1}-\eqref{1.2} in the sense of Definition \ref{def1}, if and only if the $\L^p([0,1])$-valued $\mathscr{F}_t$-adapted locally bounded stochastic process $u_\e(\cdot)$ has a continuous modification satisfying the following integral equation:
	\begin{align}\label{WF}\nonumber
		\int_0^1u_\e(t,y)\phi(y)\d y& =\int_0^1u^0(y)\phi(y)\d y +\nu\int_0^t\int_0^1u_\e(s,y)\phi''(y)\d y\d s\\&\nonumber\quad+\frac{\alpha}{\delta+1}\int_0^t\int_0^1 \p(u_\e(s,y))\phi'(y)\d y\d s+\beta \int_0^t\int_0^1 \c(u_\e(s,y))\phi(y)\d y\d s\\&\quad +\sqrt{\e}\int_0^t\int_0^1 g(s,y,u_\e(s,y))\phi(y)\W^\Q(\d s,\d y), 
	\end{align}
	for every test function $\phi\in\C^2([0,1]), \ \phi(x)=0, \ x\in \{0,1\}$ and all $t\in[0,T]$.

	If we take  $\e=0$ in the problem \eqref{1.1}-\eqref{1.2}, then we obtain the following deterministic problem:
	\begin{align}\label{det1.1}
		\frac{\partial u_0(t,x)}{\partial t}&=\nu\frac{\partial^2 u_0(t,x)}{\partial x^2}- \alpha u_0^\delta(t,x) \frac{\partial u_0(t,x)}{\partial x}
		+\beta u_0(t,x)(1-u_0^\delta(t,x))(u_0^\delta(t,x)-\gamma),~~~~~
	\end{align}
	for $t>0,~ x\in (0,1),$ and 	\begin{align}\label{det1.2}
		u_0(t,0)=u_0(t,1)=0,\ \ t\geq0, \ \ \text{ and } \ \ u_0(0,x)=u^0(x),
	\end{align}
	which is well-posed (see \cite[Proposition 3.4]{AKMTM3} or \cite[Theorem 2.10]{KKM}) and whose mild solution is given by 
	\begin{align}\label{1.4}
		u_0(t)&=\int_0^1G(t,x,y)u^0(y)\d y+\frac{\alpha}{\delta+1}\int_0^t\int_0^1\frac{\partial G}{\partial y}(t-s,x,y)\p(u_0(s,y))\d y\d s\\&\nonumber\quad +\beta\int_0^t\int_0^1G(t-s,x,y)\c(u_0(s,y))\d y\d s, \ \text{ for all } \ t\in[0,T].
	\end{align}
	Let us recall the following well-posedness results for the integral equations \eqref{WF} and \eqref{det1.1}:
	\begin{theorem}[{\cite[Theorem 2.10]{KKM}}]\label{thrmE1}
		Assume that Hypothesis \ref{hyp1} holds and the initial data $u^0\in\L^p([0,1])$, for $p>\max\{6,2\delta+1\}$. Then, there exists a unique $\L^p([0,1])$-valued $\mathscr{F}_t$-adapted continuous process $u_\e(\cdot)$    satisfying the integral equation \eqref{mild1} such that 
		\begin{align}\label{E30}
			\sup_{\e\in(0,1]}	\E\bigg[\sup_{t\in[0,T]}\|u_\e(t)\|_{\L^p}^p\bigg]\leq C<\infty.
		\end{align}
	\end{theorem}
	\begin{theorem}[{\cite[Proposition 3.5]{AKMTM3}}]\label{thrmE2}
		Assume the initial data $u^0\in\L^p([0,1])$, for $p\geq 2\delta+1$. Then, there exists a unique $\L^p([0,1])$-valued continuous function $u_0(\cdot)$    satisfying the integral equation \eqref{1.4} such that 
		\begin{align}\label{E31}
			\sup_{t\in[0,T]}\|u_0(t)\|_{\L^p}^p\leq C<\infty.
		\end{align}
	\end{theorem}
	Let us recall a useful result from \cite{ICAM2}.
	\begin{lemma}[{\cite[Lemma 4.1]{ICAM2}}]\label{lem4.1.1}
		Let $\Psi(t):=\Psi(t,\omega)$ be a random a.s.~continuous non-decreasing process on $[0,T]$. Let $\tau^\rho =\inf\{t\geq 0: \Psi(t)\geq 
		\rho\}\wedge T$. Then $\P\big(\Psi(T)\geq \rho\big)=\P\big(\Psi(\tau^\rho)\geq \rho\big)$. Let $\tau^M$ be a stopping time such that $0\leq \tau^M\leq T$ and $\P\big(\tau^M<T\big)\leq \e$. Then 
		\begin{align}
			\P\big(\Psi(T)\geq \rho\big) \leq \P\big(\Psi(\tau^\rho\wedge \tau^M)\geq \rho\big)+\e.
		\end{align}
	\end{lemma}

	\section{Central limit theorem}\label{Sec3}\setcounter{equation}{0}
	The purpose of this section is to demonstrate the Central Limit Theorem (CLT). Let us first introduce the following sequence of stopping times:
	
	\begin{align}\label{stoppingtime}
		\tau^{M,\e}:=\inf\{t\geq 0: \|u_\e(t)\|_{\L^p}> M\}\wedge T, \ \ M\in\N,
	\end{align}
	where $p >\{6,2\delta+1\}.$ The definition of the stopping time given in \eqref{stoppingtime} along with Theorem \ref{thrmE1} yields 
	\begin{align}\label{ST}
		\lim_{M\rightarrow \infty}\sup_{\e\in (0,1]}\P\big(\tau^{M,\e}< T\big)=0.
	\end{align}
	The above argument can be justified in the following way: 
	\begin{align*}
		\E\left[\|u_\e(t)\|_{\L^p}^p\right]&=\E\left[\|u_\e(t)\|_{\L^p}^p\chi_{\{\tau^{M,\e}<T\}}\right]+\E\left[\|u_\e(t)\|_{\L^p}^p\chi_{\{\tau^{M,\e}=T\}}\right]\geq \E\left[\|u_\e(t)\|_{\L^p}^p\chi_{\{\tau^{M,\e}<T\}}\right]\nonumber\\&\geq M^p\mathbb{E}\left[\chi_{\{\tau^{M,\e}<T\}}\right]=M^p\mathbb{P}\left(\tau^{M,\e}<T\right),
	\end{align*}
	so that 
	\begin{align*}
		\sup_{\e\in (0,1]}\mathbb{P}\big(\tau^{M,\e}<T\big)\leq \sup_{\e\in (0,1]} \frac{1}{M^p}\E\left[\|u_\e(t)\|_{\L^p}^p\right]\leq \frac{C}{M^p}\to 0, \ \text{ as }\ M\to\infty. 
	\end{align*}
	Note that the stopping time defined in \eqref{stoppingtime} plays a crucial role in the further analysis of this section. 
	\begin{proposition}\label{prop1}
		Under the  assumption \ref{hyp1}, for any $p> \max\{6,2\delta+1\}$, the following holds:
		\begin{align}\label{prop3.1.1}
			\E\bigg[\sup_{t\in[0,T\wedge \tau^{M,\e}]} \|u_\e(t)-u_0(t)\|_{\L^p}^p\bigg]\leq \e^{\frac {p}{2}} C(L,M, T, \delta, \gamma,\beta) \to 0 \ \text{ as }\ \e\to0.
		\end{align}
	\end{proposition}
	\begin{proof} From \eqref{mild1} and \eqref{1.4}, we have for all $t\in[0,T]$, $\P$-a.s.,
		\begin{align}\label{2.18}
			u_\e(t,x)-u_0(t,x)&=
			\nonumber \frac{\alpha}{\delta+1}\int_0^t\int_0^1\frac{\partial G}{\partial y}(t-s,x,y)(\p(u_\e(s,y))-\p(u_0(s,y)))\d y\d s\\ \nonumber &\quad +\beta\int_0^t\int_0^1G(t-s,x,y)(\c(u_\e(s,y))-\c(u_0(s,y)))\d y\d s\\\nonumber &\quad+ \sqrt{\e}
			\int_0^t\int_0^1G(t-s,x,y)g(s,y,u_\e(s,y))\W^\Q(\d s,\d y) \\& =:\frac{\alpha}{\delta+1}I_1(t,x)+\beta I_2(t,x)+\sqrt{\e}I_3(t,x). 
		\end{align} Taking $\L^p$-norm and then applying  \eqref{A2}, \eqref{A7}, Minkowski's, H\"older's and Young's inequalities , we estimate the term $\|I_1(\cdot)\|_{\L^p}$ as
		\begin{align}\label{I_1}\nonumber
			\|I_1(t)\|_{\L^p}&=\bigg\|\int_0^t\int_0^1\frac{\partial G}{\partial y}(t-s,x,y)(\p(u_\e(s,y))-\p(u_0(s,y)))\d y\d s\bigg\|_{\L^p}\\
			&\nonumber \leq C \int_0^t (t-s)^{-1}\big\|e^{\frac{-|\cdot|^2}{a_2(t-s)}}*\big|\p(u_\e(s,\cdot))-\p(u_0(s,\cdot))\big|\big\|_{\L^p}\d s\\
			&\nonumber\leq C \int_0^t (t-s)^{-1}\big\|e^{\frac{-|\cdot|^2}{a_2(t-s)}}\big\|_{\L^p} \|\p(u_\e(s))-\p(u_0(s))\|_{\L^1}\d s
			\\& \leq C \int_0^t (t-s)^{-1+\frac{1}{2p}} \left(\|u_\e(s)\|_{\L^{\frac{p\delta}{p-1}}}^\delta+\|u_0(s)\|_{\L^{\frac{p\delta}{p-1}}}^\delta\right)\|u_\e(s)-u_0(s)\|_{\L^p}\d s.
		\end{align} where the last inequality is obtained  by using  Taylor's formula (see \cite[Theorem 7.9.1]{PGC})
		\begin{align*}
			|\p(u_\e)-\p(u_0)|&\leq |(u_\e-u_0)\p'(u_\e+\theta (u_\e-u_0) )|\nonumber\leq 2^{\delta-1}(\delta+1)(|u_\e|^\delta+ |u_0|^\delta)|u_\e-u_0|,
		\end{align*}
		for some constant $\theta\in (0,1)$  	and H\"{o}lder's inequality. Using one more time H\"{o}lder's inequality in \eqref{I_1}, we find
		\begin{align*}
			\|I_1(t)\|_{\L^p}^p\leq& C\bigg[\int_0^t (t-s)^{-1+\frac{1}{2p}} \bigg(\|u_\e(s)\|_{\L^{\frac{p\delta}{p-1}}}^{\frac{p\delta}{p-1}}+\|u_0(s)\|_{\L^{\frac{p\delta}{p-1}}}^{\frac{p\delta}{p-1}}\bigg)\d r\bigg]^{p-1}\\&\quad 		\times \bigg[\int_0^t (t-s)^{-1+\frac{1}{2p}}\|u_\e(s)-u_0(s)\|_{\L^p}^p\d s\bigg].
		\end{align*}
		Taking supremum of time over  $[0,T\wedge \tau^{M,\e}]$, and then taking expectation on both sides, we get 
		\begin{align*}
			&	\E\bigg[\sup_{t\in [0,T\wedge \tau^{M,\e}]}\|I_1(t)\|_{\L^p}^p\bigg]\\&\leq C \E\bigg[\sup_{t\in [0,T\wedge \tau^{M,\e}]}\bigg\{\bigg(\int_0^t (t-s)^{-1+\frac{1}{2p}} \bigg(\|u_\e(s)\|_{\L^{\frac{p\delta}{p-1}}}^{\frac{p\delta}{p-1}}+\|u_0(s)\|_{\L^{\frac{p\delta}{p-1}}}^{\frac{p\delta}{p-1}}\bigg)\d s\bigg)^{p-1}
			\\&\quad\quad \times \bigg(\int_0^t (t-s)^{-1+\frac{1}{2p}}\|u_\e(s)-u_0(s)\|_{\L^p}^p\d s\bigg)\bigg\}\bigg]
			\\&\quad \leq  C(p,M,T, \delta)\int_0^{T} (t-s)^{-1+\frac{1}{2p}}\E\bigg[\sup_{r \in [0,s\wedge \tau^{M,\e}]}\|u_\e(r)-u_0(r)\|_{\L^p}^p\bigg]\d s,
		\end{align*}
		where we have used stopping time arguments and the bound of $u_0$ from Theorem   \ref{thrmE2} .
		
		Using  \eqref{A1}, \eqref{A7}, Minkowski's, H\"older's and Young's inequalities we estimate the term  $\|I_2(\cdot)\|_{\L^p}$, as
		\begin{align}\label{I2}\nonumber
			\|I_2(t)\|_{\L^p}&\leq C \int_0^t (t-s)^{-\frac12}\big\|e^{\frac{-|\cdot|^2}{a_2(t-s)}}*\big|\c(u_\e(s))-\c(u_0(s))\big|\big\|_{\L^p}\d s
			\\&\nonumber
			\leq C \int_0^t (t-s)^{-\frac12}\big\|e^{\frac{-|\cdot|^2}{a_2(t-s)}}\big\|_{\L^p}\|\c(u_\e(s))-\c(u_0(s))\|_{\L^1}\d s
			\\&\leq C \int_0^t (t-s)^{-\frac12+\frac{1}{2p}}\|\c(u_\e(s))-\c(u_0(s))\|_{\L^1}\d s.
		\end{align}
		Using Taylor's formula, we have the following inequality:
		\begin{align*}
			\big\|\c(u_\e)-\c(u_0)\big\|_{\L^1}&\leq (1+\gamma)\|u^{\delta+1}_\e-u_0^{\delta+1}\|_{\L^1}+\gamma\|u_\e-u_0\|_{\L^1}+\|u^{2\delta+1}_\e-u_0^{2\delta+1}\|_{\L^1}
			\\&\leq C(\delta, \gamma)\bigg\{1+\|u_\e\|^{\delta}_{\L^{\frac{\delta p}{p-1}}}+\|u_\e\|^{2\delta}_{\L^{\frac{2p\delta}{p-1}}}+\|u_0\|^{\delta}_{\L^{\frac{p\delta}{p-1}}}+\|u_0\|^{2\delta}_{\L^{\frac{2p\delta}{p-1}}}\bigg\}\|u_\e-u_0\|_{\L^p}.
		\end{align*}
		Proceeding in the same way as in the case of $I_1$, that is, by using H\"{o}lder's inequality, taking the supremum over time from $0$ to $T\wedge \tau^{M,\e}$,   taking expectation and then  using stopping arguments and the estimate \eqref{E31}  from Theorem \ref{thrmE2}, we get
		\begin{align*}
			\E\bigg[\sup_{t\in [0,T\wedge \tau^{M,\e}]}\|I_2(t)\|_{\L^p}^p\bigg]	\leq  C(M,T,\delta, \gamma)\int_0^{T} (t-s)^{-\frac{1}{2}+\frac{1}{2p}}\E\bigg[\sup_{r\in [0,s\wedge \tau^{M,\e}]}\|u_\e(r)-u_0(r)\|_{\L^p}^p\bigg]\d s.
		\end{align*}
		Finally, using Hypothesis \ref{hyp1}, Fubini's theorem, BDG (cf. \eqref{18}), H\"older's,   Minkowski's inequalities, \eqref{A1}, \eqref{A7} and Young's   inequality for convolution, we estimate the term \break$\E\bigg[\sup\limits_{t\in [0,T\wedge \tau^{M,\e}]}	\|I_3(t)\|_{\L^p}^p\bigg]$ as
		\begin{align}\label{noiseEstimate}\nonumber
			&\E\bigg[\sup_{t\in [0,T\wedge \tau^{M,\e}]}	\|I_3(t)\|_{\L^p}^p\bigg] \\&\nonumber=  \E\bigg[\sup_{t\in [0,T\wedge \tau^{M,\e}]}\bigg\|	\int_0^{t}\int_0^1G(t-s,\cdot,y)g(s,y,u_{\e}(s,y))\W^\Q(\d s,\d y) \bigg\|_{\L^p}^p\bigg]\\& \nonumber=  \E\bigg[\sup_{t\in [0,T\wedge \tau^{M,\e}]}\bigg\|\sum_{j\in\N}	\int_0^{t}\bigg(\int_0^1 G(t-s,\cdot,y)g(s,y,u_{\e}(s,y))q_j\varphi_j(y)\d y\bigg)\d \beta_j(s) \bigg\|_{\L^p}^p\bigg]	\\&\nonumber\leq C   \E\bigg[\bigg\|\bigg(\sum_{j\in\N}	\int_0^{T\wedge \tau^{M,\e}}\bigg|\int_0^1 q_j\varphi_j(y)G(t-s,\cdot,y)
			g(s,y,u_{\e}(s,y)\d y\bigg|^2\d s\bigg)^{\frac{1}{2}}\bigg\|_{\L^p}^p\bigg]	\nonumber\\&\leq C   
			\E\bigg[\bigg\|\bigg(\sum_{j\in\N}\int_0^1q_j^2\varphi_j^2(y)\d y\bigg(\int_0^{T\wedge \tau^{M,\e}}\int_0^1G^2(t-s,\cdot,y)g^2(s,y,u_{\e}(s,y))\d y\d s\bigg)\bigg)^{\frac{1}{2}}\bigg\|_{\L^p}^p\bigg]\nonumber\\&\nonumber \leq C\bigg(\sum_{j\in\N}q_j^2\bigg)^\frac{p}{2}  \E\bigg[\bigg\|\int_0^{T\wedge \tau^{M,\e}}\int_0^1G^2(t-s,\cdot,y)g^2(s,y,u_{\e}(s,y))\d y\d s\bigg\|_{\L^{\frac{p}{2}}}^{\frac{p}{2}}\bigg]
			\\&\nonumber \leq C\E\bigg[\bigg(\int_0^{T\wedge \tau^{M,\e}}(t-s)^{-1}\big\|e^{-\frac{|\cdot|^2}{a_1(t-s)}}*g^2(s,\cdot,u_{\e}(s,\cdot))\big\|_{\L^{\frac{p}{2}}}\d s\bigg)^{\frac{p}{2}}\bigg] \nonumber\\&\nonumber\leq  CK  T^{\frac{p-2}{4}}\E\bigg[ \int_0^{T\wedge \tau^{M,\e}}(t-s)^{-\frac{1}{2}}\big(1+\|u_{\e}(s)\|_{\L^p}^p\big)\d s\bigg]
			\\&\leq CK \big(1+M^p\big) T^\frac{p}{4},
		\end{align}where we have used the fact that $\{\varphi_j\}_{j\in\N}$ is an orthonormal family on $\L^2([0,1])$, $\sum\limits_{j\in\N}q_j^2<\infty$ and stopping time arguments.

		Combining the above estimates in \eqref{2.18}, we deduce
		\begin{align*}
			&\E\bigg[\sup_{t\in [0,T\wedge \tau^{M,\e}]} \|u_\e(t)-u_0(t)\|_{\L^p}^p\bigg]\\&\leq  \e^{\frac {p}{2}}CK(1+M^p)T^\frac{p}{4}+C(M, T, \delta, \gamma,\alpha,\beta)
			\int_0^T \big\{(t-s)^{-1+\frac{1}{2p}}+ (t-s)^{-\frac{1}{2}+\frac{1}{2p}}\big\}\\&\quad\times\E\bigg[\sup_{r\in [0,s\wedge \tau^{M,\e}]}\|u_\e(r)-u_0(r)\|_{\L^p}^p\bigg]\d s.
		\end{align*}
		An application of Gronwall’s inequality yields
		\begin{align}\label{CLT01}\E\bigg[\sup_{t\in [0,T\wedge \tau^{M,\e}]} \|u_\e(t)-u_0(t)\|_{\L^p}^p\bigg]\leq\e^{\frac {p}{2}} C(K,M, T, \delta, \gamma,\beta),
		\end{align}
		and the proof is completed.
	\end{proof}
	\begin{theorem}[Central limit theorem]\label{thrm2.2} 
		Under the assumption (\ref{hyp1}), for any $T>0$, the process $$v_\e:=\frac{u_\e-u_0}{\sqrt{\e}}$$ converges to a random field $v$ in probability on the space $\C([0,T];\L^p([0,1])),$ for any $p> \max\{6,2\delta+1\},$ where the random field $v$ satisfies the following integral equation for all $t\in[0,T]$, $\P$-a.s.,
		\begin{align}\label{2.21}
			v(t,x)&=\nonumber \frac{\alpha}{\delta+1}\int_0^t\int_0^1\frac{\partial G}{\partial y}(t-s,x,y)\p'(u_0(s,y))v(s,y)\d y\d s\\ \nonumber&\quad +\beta\int_0^t\int_0^1G(t-s,x,y)\c'(u_0(s,y))v(s,y)\d y\d s\\   &\quad+ \sqrt{\e}
			\int_0^t\int_0^1G(t-s,x,y)g(s,y,u_0(s,y))\W^\Q(\d s,\d y),
		\end{align}
		where $\p'(u_0)=(\delta+1)u_0^{\delta}$ and $\c'(u_0)=-\gamma+(1+\gamma)(1+\delta)u_0^{\delta}-(2\delta+1)u_0^{2\delta}$. 
	\end{theorem}
	\begin{remark}
		With the regularity of $u_0\in \C([0,T];\L^p([0,1]))$, for $p\geq 2\delta+1$ (see Theorem \ref{thrmE2}), 	the well-posedness of the problem \eqref{2.21} follows in the same lines as in the proof of \cite[Proposition 3.5]{AKMTM3}. Moreover, we have 
		\begin{align}\label{3p10}
			\sup_{t\in[0,T]}\|v(t)\|_{\L^p}^p\leq C<\infty.
		\end{align}
	\end{remark}
	\begin{proof}[Proof of Theorem \ref{thrm2.2} ]
		Let us set  $v_\e:=\frac{u_\e-u_0}{\sqrt{\e}}$. In order to prove this theorem, we need to show that for any $\varpi>0,$
		\begin{align}\label{CLT1}
			\lim_{\e\rightarrow 0}\P\bigg(\sup_{t\in[0,T]}\|v_\e(t)-v(t)\|_{\L^p}\geq \varpi\bigg)=0,
		\end{align}which implies the convergence in probability and hence in distribution.
		In view of \eqref{CLT1}, we need to establish the following: 
		\begin{align*}
			\lim_{\e\rightarrow 0}\P\bigg(\sup_{t\in[0,T\wedge \tau^{M,\e}]}\|v_\e(t)-v(t)\|_{\L^p}\geq \varpi\bigg)=0,
		\end{align*} for any $M>0$ large enough. The above fact can be justified by using Lemma \ref{lem4.1.1} and \eqref{ST} as follows:
		\begin{align}\label{3.11}\P\bigg(\sup_{t\in[0,T]}\|v_\e(t)-v(t)\|_{\L^p}\geq \varpi\bigg)&\leq \P\bigg(\sup_{t\in[0,T\wedge \tau^{M,\e}]}\|v_\e(t)-v(t)\|_{\L^p}\geq \varpi\bigg)+\P(\tau^{M,\e}< T)\nonumber\\& \leq \frac{1}{\varpi} \E\bigg[\sup_{t\in[0,T\wedge \tau^{M,\e}]}\|v_\e(t)-v(t)\|_{\L^p}\bigg]+\P(\tau^{M,\e}< T),
		\end{align}where we have used Markov's inequality and Theorems \ref{thrmE1} and \ref{thrmE2}.

		By the definition of $v_\e$ and $v$, we have
		\begin{align}\label{2.22}
			&	v_\e(t,x)-\nonumber v(t,x)
			\\ \nonumber &	=\frac{\alpha}{\delta+1}\int_0^t\int_0^1\frac{\partial G}{\partial y}(t-s,x,y)\bigg[\frac{1}{\sqrt{\e}} \big(\p(u_\e(s,y))-\p(u_0(s,y))\big)-\p'(u_0(s,y))v(s,y)\bigg]\d y\d s 	\\ \nonumber &\quad +\beta\int_0^t\int_0^1G(t-s,x,y)\bigg[\frac{1}{\sqrt{\e}} \big(\c(u_\e(s,y))-\c(u_0(s,y))\big)-\c'(u_0(s,y))v(s,y)\bigg]\d y\d s	\\ \nonumber &\quad+ 
			\sqrt{\e}\int_0^t\int_0^1G(t-s,x,y)\big(g(s,y,u_\e(s,y))-g(s,y,u_0(s,y))\big)\W^\Q(\d s,\d y)
			\\& =:\frac{\alpha}{\delta+1}J_1^\e(t,x)+\beta J_2^\e(t,x)+\sqrt{\e}J_3^\e(t,x).
		\end{align}
		Using the fact that $v_\e=\frac{u_\e-u_0}{\sqrt{\e}}$, we can re-write the term $J_2^\e$ appearing in the right hand side of \eqref{2.22} as
		{\small	\begin{align}\label{2.23}\nonumber
				&	J_2^\e(t,x)\\&=\nonumber\int_0^t\int_0^1G(t-s,x,y)\bigg[\frac{1}{\sqrt{\e}}\big( \c(u_0(s,y)+\sqrt{\e} v_\e(s,y))-\c(u_0(s,y))\big)-\c'(u_0(s,y))v_\e(s,y)\bigg]\d y\d s
				\\\nonumber&\quad +\int_0^t\int_0^1G(t-s,x,y)\bigg[\c'(u_0(s,y))(v_\e(s,y)-v_0(s,y))\bigg]\d y \d s
				\\&=:J_{21}^\e(t,x)+ J_{22}^\e(t,x).
		\end{align}}
		Using Taylor's formula, there exists an $\eta^\e\in(0,1)$ such that 
		\begin{align}\label{CLT2}
			\c(u_0+\sqrt{\e} v_\e)-\c(u_0)=\sqrt{\e}\c'(u_0+\sqrt{\e}\eta^{\e}v_\e)v_\e.
		\end{align}
		Further, there exists a  $\kappa^\e\in(0,1)$, such that 
		\begin{align}\label{CLT3}	\nonumber
			&	\bigg|\bigg(\frac{1}{\sqrt{\e}}\big( \c(u_0+\sqrt{\e} v_\e)-\c(u_0)\big)\bigg)-\c'(u_0)v_\e\bigg|\\&\nonumber\quad
			=\big|\big(\c'(u_0+\sqrt{\e}\eta^{\e}v_\e)-\c'(u_0)\big)v_\e \big|\\&\nonumber\quad 	=\sqrt{\e}\eta^\e|\c''((u_0+\sqrt{\e}\kappa^{\e}v_\e))||v_\e|^2
			\\&\nonumber\quad=\sqrt{\e}\eta^\e \left|\delta(\delta+1)(\gamma+1)(u_0+\sqrt{\e}\kappa^{\e}v_\e)^{\delta-1}
			-2\delta(2\delta+1)(u_0+\sqrt{\e}\kappa^{\e}v_\e)^{2\delta-1}\right||v_\e|^2
			\\&\nonumber\quad \leq  C(\delta,\gamma)\left\{\left(\e^{\frac{1}{2}}|u_0|^{\delta-1}+\e^{\frac12}|u_0|^{2\delta-1}\right)|v_\e|^2
			+\left( \e^{\frac{\delta}{2}}|v_\e|^{\delta+1}+\e^{\delta}|v_\e|^{2\delta+1}\right)\right\}
			\\&\quad \leq  (\e^{\frac12}+\e^{\delta}+\e^{\frac{\delta}{2}})C(\delta,\gamma) \left\{\left(|u_0|^{\delta-1}+|u_0|^{2\delta-1}\right)|v_\e|^2
			+\left( |v_\e|^{\delta+1}+|v_\e|^{2\delta+1}\right)\right\},
		\end{align}
		where $\c''(u_0)=(1+\gamma)\delta(1+\delta)u_0^{\delta-1}-2\delta(2\delta+1) u_0^{2\delta-1}$. Now, we consider the term $\|J_{21}^\e(\cdot)\|_{\L^p}$ and estimate it using a similar calculation to \eqref{I2} along with \eqref{CLT2} and \eqref{CLT3}, to find 
		\begin{align*}
			&\|J_{21}^\e(t)\|_{\L^p}\\&\quad \leq (\e^{\frac12}+\e^{\delta}+\e^{\frac{\delta}{2}})C(\delta,\gamma)\bigg[\int_{0}^{t}(t-s)^{-\frac12+\frac{1}{2p}}\bigg(\big\|\big(|u_0(s)|^{\delta-1}+|u_0(s)|^{2\delta-1}\big)|v_\e(s)|^2\big\|_{\L^1}
			\\&\qquad\quad +\big\| |v_\e(s)|^{\delta+1}+|v_\e(s)|^{2\delta+1}\big\|_{\L^{1}}\bigg)\d s\bigg]
			\\&\quad \leq (\e^{\frac12}+\e^{\delta}+\e^{\frac{\delta}{2}})C(\delta,\gamma)\bigg[\int_{0}^{t}(t-s)^{-\frac12+\frac{1}{2p}}\bigg\{\bigg(\|u_0(s)\|_{\L^{\frac{p(\delta-1)}{p-2}}}^{\delta-1}+\|u_0(s)\|_{\L^{\frac{p(2\delta-1)}{p-2}}}^{2\delta-1}\bigg)	\|v_\e(s)\|_{\L^p}^{2}
			\\&\qquad\quad+\bigg( \|v_\e(s)\|_{\L^{\delta+1}}^{\delta+1}+\|v_\e(s)\|_{\L^{2\delta+1}}^{2\delta+1}\bigg)\bigg\}\d s\bigg].
		\end{align*}
		Taking supremum from $0$ to  $T\wedge \tau^{M,\e},$ and then taking expectation on both sides of above inequality, we deduce 
		\begin{align*}
			\E\bigg[\sup_{t\in [0,T\wedge \tau^{M,\e}]}\|J_{21}^\e(t)\|_{\L^p}\bigg]
			\leq(\e^{\frac12}+\e^{\delta}+\e^{\frac{\delta}{2}}) C( M,\delta,\gamma,T).
		\end{align*}
		Next, we  estimate  the term  $J_{22}^\e$ as
		\begin{align*}
			\|J_{22}^\e(t)\|_{\L^p}&\leq C\int_{0}^{t}(t-s)^{-\frac12+\frac{1}{2p}}\|\c'(u_0(s))
			(v_\e(s)-v(s))\|_{\L^1}\d s
			\\&\leq C\int_{0}^{t}(t-s)^{-\frac12+\frac{1}{2p}}\|\c'(u_0(s))\|_{\L^{\frac{p}{p-1}}}\|v_\e(s)-v(s)\|_{\L^p}\d s.
		\end{align*}
		Taking supremum from $0$ to  $T\wedge \tau^{M,\e}$, and then taking expectation on both sides of above inequality, we get
		\begin{align}\label{J_22}\nonumber
			\E&\bigg[\sup_{t\in [0,T\wedge \tau^{M,\e}]}\|J_{22}^\e(t)\|_{\L^p}\bigg]
			\\&\quad\leq C\int_{0}^{T}(t-s)^{-1+\frac{1}{2p}}\|\c'(u_0(s))\|_{\L^{\frac{p}{p-1}}}
			\E\bigg[\sup_{r\in [0,s\wedge \tau^{M,\e}]}\|v_\e(r)-v(r)\|_{\L^p}\bigg]	\d s.
		\end{align}
		Combining the estimates of $J_{21}^\e$ and $J_{22}^\e$ in \eqref{2.23}, we obtain 
		\begin{align*}
			&\E\bigg[\sup_{t\in [0,T\wedge \tau^{M,\e}]}\|J_{2}^\e(t)\|_{\L^p}\bigg]
			\\&\leq (\e^{\frac12}+\e^{\delta}+\e^{\frac{\delta}{2}})C(\delta,\gamma,T)	\\&\quad +	C\int_{0}^{T}(t-s)^{-\frac12+\frac{1}{2p}}\|\c'(u_0(s))\|_{\L^{\frac{p}{p-1}}}
			\E\bigg[\sup_{r\in [0,s\wedge \tau^{M,\e}]}\|v_\e(s)-v(s)\|_{\L^p}\bigg]\d s.	
		\end{align*}
		In a similar way, we estimate the term $J_{1}^\e$ as
		\begin{align}\label{PJ_1}\nonumber
			&\E\bigg[\sup_{t\in [0,T\wedge \tau^{M,\e}]}\|J_{1}^\e(t)\|_{\L^p}\bigg]
			\\&\leq \e^{\frac{\delta}{2}}C(\delta,\gamma,T)+	C\int_{0}^{T}(t-s)^{-\frac12+\frac{1}{2p}}\|\p'(u_0(s))\|_{\L^{\frac{p}{p-1}}}
			\E\bigg[\sup_{r\in [0,s\wedge \tau^{M,\e}]}\|v_\e(r)-v(r)\|_{\L^p}\bigg]	\d s.	
		\end{align}Let us consider the term $\E\bigg[\sup\limits_{t\in[0,T\wedge\tau^{M,\e}]}\|J_3^\e(t)\|_{\L^p}^p\bigg]$, and estimate it using Hypothesis \ref{hyp1}, Fubini's theorem, BDG (cf. \eqref{18}), H\"older's,   Minkowski's inequalities, \eqref{A1}, \eqref{A7}, Young's   inequality for convolution and \eqref{CLT01} (see Proposition \ref{prop1}) as 
		\begin{align}\label{noiseEstimate1}\nonumber
			&\E\bigg[\sup_{t\in [0,T\wedge \tau^{M,\e}]}	\|J_3(t)\|_{\L^p}^p\bigg] 
			\\& \nonumber=  \E\bigg[\sup_{t\in[0,T\wedge \tau^{M,n}]}\bigg\|\sum_{j\in\N}	\int_0^{t}\bigg(\int_0^1 G(t-s,\cdot,y)\\&\nonumber\qquad\quad\times\big(g(s,y,u_{\e}(s,y))-g(s,y,u_{0}(s,y))\big)q_j\varphi_j(y)\d y\bigg)\d \beta_j(s) \bigg\|_{\L^p}^p\bigg]	\\&\nonumber\leq C   \E\bigg[\bigg\|\bigg(\sum_{j\in\N}	\int_0^{T\wedge \tau^{M,\e}}\big|\big(G(t-s,\cdot,y)
			(g(s,y,u_{\e}(s,y))-g(s,y,u_{0}(s,y))),q_j\varphi_j(y)\big)\big|^2\d s\bigg)^{\frac{1}{2}}\bigg\|_{\L^p}^p\bigg]	
			\\&\nonumber \leq\bigg(\sum_{j\in\N}q_j^2\bigg)^{\frac{p}{2}} CL  \E\bigg[\bigg\|\int_0^{T\wedge \tau^{M,\e}}\int_0^1G^2(t-s,\cdot,y)\big|u_{\e}(s,y)-u_{0}(s,y)\big|^2\d y\d s\bigg\|_{\L^{\frac{p}{2}}}^{\frac{p}{2}}\bigg]
			\\&\nonumber \leq CL\E\bigg[\bigg(\int_0^{T\wedge \tau^{M,\e}}(t-s)^{-1}\big\|e^{-\frac{|\cdot|^2}{a_1(t-s)}}*\big|(u_{\e}-u_{0})^2(s))\big|\big\|_{\L^{\frac{p}{2}}}\d s\bigg)^{\frac{p}{2}}\bigg] \nonumber
			\\&\leq  CL  T^{\frac{p-2}{4}}\E\bigg[ \int_0^{T\wedge \tau^{M,\e}}(t-s)^{-\frac{1}{2}}\|u_{\e}(s)-u_{0}(s)\|_{\L^p}^p\d s\bigg]\nonumber
			\\&\leq  CL  T^{\frac{p-2}{4}}  \int_0^{T}(t-s)^{-\frac{1}{2}}   \E\bigg[\sup_{r\in[0,s\wedge \tau^{M,\e}]}\|u_{\e}(r)-u_{0}(r)\|_{\L^p}^p\bigg]\d s
			\nonumber\\&\leq \e^{\frac {p}{2}} C(L,M, T, \delta, \gamma,\beta).
		\end{align}
		Substituting the above estimates in \eqref{2.22}, we find
		\begin{align}\label{CLT4}\nonumber
			&\E\bigg[\sup_{t\in [0,T\wedge \tau^{M,\e}]}\|v_\e(t)-v(t)\|_{\L^p}\bigg]
			\\&\nonumber \leq \big(\e^\frac{1}{2}+\e^\delta+\e^\frac{\delta}{2}\big)C(M,\delta,\gamma,T)\\&\nonumber\quad +C\int_0^T\bigg((t-s)^{-1+\frac{1}{2p}}\|\c'(u_0(s))\|_{\L^{\frac{p}{p-1}}}+(t-s)^{-\frac{1}{2}-\frac{1}{2p}}\|\p'(u_0(s))\|_{\L^{\frac{p}{p-1}}}\bigg)\\&\nonumber\qquad\quad\times	\E\bigg[\sup_{r\in [0,s\wedge \tau^{M,\e}]}\|v_\e(r)-v(r)\|_{\L^p}\bigg]	\d s+\e^{\frac{p+1}{2}}C(L,M,\delta,\gamma,\beta,T)
			\\&\nonumber \leq \big(\e^\frac{1}{2}+\e^\delta+\e^\frac{\delta}{2}+\e^{\frac{p+1}{2}}\big)C(L,M,\delta,\gamma,T)\\&\nonumber\quad +C\int_0^T\bigg((t-s)^{-1+\frac{1}{2p}}\|\c'(u_0(s))\|_{\L^{\frac{p}{p-1}}}+(t-s)^{-\frac{1}{2}-\frac{1}{2p}}\|\p'(u_0(s))\|_{\L^{\frac{p}{p-1}}}\bigg)\\&\qquad\quad\times	\E\bigg[\sup_{r\in [0,s\wedge \tau^{M,\e}]}\|v_\e(r)-v(r)\|_{\L^p}\bigg]	\d s.
		\end{align}
		An application of Gronwall's inequality in \eqref{CLT4} yields
		\begin{align}\label{CLT5}\nonumber
			&\E\bigg[\sup_{t\in [0,T\wedge \tau^{M,\e}]}\|v_\e(t)-v(t)\|_{\L^p}\bigg] 	\\&\nonumber \leq \big(\e^\frac{1}{2}+\e^\delta+\e^\frac{\delta}{2}+\e^{\frac{p+1}{2}}\big)C\\&\quad\times \exp\bigg\{C\int_0^T\bigg((t-s)^{-1+\frac{1}{2p}}\|\c'(u_0(s))\|_{\L^{\frac{p}{p-1}}}+(t-s)^{-\frac{1}{2}-\frac{1}{2p}}\|\p'(u_0(s))\|_{\L^{\frac{p}{p-1}}}\bigg)\d s\bigg\},
		\end{align}where the term in the exponential in the right hand side of the above inequality is finite by using Theorem \ref{thrmE2} (see \eqref{E31}). 
		
		Finally, passing $\e\to0$ in \eqref{CLT5} leads to 
		\begin{align*}
			\lim_{\e\to0} \E\bigg[\sup_{t\in [0,T\wedge \tau^{M,\e}]}\|v_\e(t)-v(t)\|_{\L^p}\bigg] =0,
		\end{align*}
		which completes the proof by using \eqref{3.11}.
	\end{proof}
	
	\section{Moderate Deviation Principle}\label{Sec4}\setcounter{equation}{0}
	In this section,  we first recall some basic definitions related to LDP  and then we state a sufficient condition for LDP to hold, followed by our main result (Theorem \ref{MDPthm}) of this section. Let $\EE$ be a Polish space.
	\begin{definition}[Rate Function] A function $I:\EE \to [0,\infty]$ is called a \emph{rate function} on $\EE$ if it is  lower semicontinuous on $\EE$. A rate function $\I(\cdot)$ on $\EE$ is said to be a good rate function on $\EE$ if for every $N\in[0,\infty)$, the level set $\mathrm{K}:=\{\psi\in\EE:\I(\psi)\leq N\}$ is a compact subset of $\EE$.
	\end{definition}
	\begin{definition}[Speed function]
		A family of positive numbers $\{\lambda(\e)\}_{\e>0},$ is called a \emph{speed function} if $\lambda(\e)\to +\infty$ and $\sqrt{\e}\lambda(\e)\rightarrow 0$, as $\e \rightarrow 0$.
	\end{definition}
	\begin{definition}[LDP]\label{LDP}
		Let $\I:\EE\to[0,\infty]$ be a rate function on $\EE$. The sequence $\{\mathrm{X}_\e\}_{\e>0}$ satisfies an LDP on the space $\EE$ with the rate function $\I(\cdot)$ and with the speed $\{\lambda(\e)\}_{\e>0},$ if the following hold:
		\begin{enumerate}
			\item For any closed subset $\F\subset\EE$, 
			\begin{align*}
				\limsup_{\e\to0}\frac{1}{\lambda(\e)}\log \P\big(\mathrm{X}_\e\in\F\big) \leq -\inf_{\psi\in\F}\I(\psi).
			\end{align*}
			\item For any open subset $\mathrm{G}\subset\EE$, 
			\begin{align*}
				\liminf_{\e\to0}\frac{1}{\lambda(\e)}\log \P\big(\mathrm{X}_\e\in\G\big) \geq -\inf_{\psi\in\mathrm{G}}\I(\psi).
			\end{align*}
		\end{enumerate}
	\end{definition}
	Our proof of LDP for the law of solution of the system \eqref{1.1}-\eqref{1.2} is based on a weak convergence approach (see \cite{ADPDVM}). From the work \cite{ADPDVM}, one can obtain the equivalency between the Laplace principle and LDP.
	
	\begin{definition}[Laplace principle]\label{LP}
		The sequence of the random variables $\{\mathrm{X}_\e\}_{\e>0}$ on $\EE$ is said to satisfy the \emph{Laplace principle} with the rate function $\I(\cdot)$ having speed $\lambda(\e)$ if for each $f\in \C_b(\EE)$ (space of bounded continuous functions from $\EE$ to $\R$), 
		\begin{align*}
			\lim_{\e\to0}\lambda(\e)\log \E\bigg[\exp\bigg\{-\frac{f(\mathrm{X}_\e)}{\lambda(\e)}\bigg\}\bigg]=-\inf_{\psi\in\EE}\big\{f(\psi)+\I(\psi)\big\}.
		\end{align*}
	\end{definition}
	
	Define the Cameron-Martin space 
	\begin{align*}
		\mathcal{H}:=\bigg\{\mathfrak{h}=(\mathfrak{h}_1,\mathfrak{h}_2,\ldots,&\mathfrak{h}_j\ldots):[0,T]\to\ell^2\big| \  \mathfrak{h} \text{ is absolutely} \\&
		\text{continuous and }\|\mathfrak{h}\|_{\mathcal{H}}^2:=\sum_{j\in\N}\int_0^{T}|\dot{\mathfrak{h}}_j(s)|^2\d s<+\infty\bigg\},
	\end{align*}
	where the dot denotes the weak derivative (for more details see \cite[Section 4]{ADPDVM} or \cite[Section 2]{MRTZXZ}). Let $\langle \cdot,\cdot \rangle_\mathcal{H}$ represent the inner product in $\mathcal{H}$. Define the following sets:
	\begin{align*}
		\UU^N &:=\bigg\{\mathfrak{h}\in\mathcal{H}: \sum_{j\in\N}\int_0^{T}|\dot{\mathfrak{h}}_j(s)|^2\d s\leq N \bigg\}, \ N\in\N,\\
		\PP_2&:=\bigg\{\mathfrak{h}\in \mathcal{H}: \mathfrak{h} \text{ is an $\mathscr{F}_t$-adapted process}\bigg\},\\
		\PP_2^N&:=\bigg\{\mathfrak{h}\in\PP_2: \ \mathfrak{h}(\omega)\in\UU^N, \ \ \P\text{-a.s.}\bigg\}, 
	\end{align*}
	where $\UU^N$ denotes a compact metric space equipped with the weak topology in $\mathcal{H}$. The set $\PP_2^N$ represents a permissible collection of controls.
	
	Consider another Polish space $\EE_0$, such that the initial data $u^0$ belongs to a compact subset of $\EE_0$.
	For every $\e>0$, we obtain the mapping $\mathscr{G}_\e:\EE_0\times\C([0,T];\R^{\infty})\to\EE$ as a collection of measurable functions. Consider $\mathrm{X}_{\e,u^0}(\cdot)=\mathscr{G}_\e\big(u^0,\sqrt{\e}\boldsymbol{\beta}(
	\cdot)\big),$ where $\boldsymbol{\beta}(\cdot):=(\beta_1(\cdot),\beta_2(\cdot),\ldots,\beta_j(\cdot),\ldots)$ represents an infinite-dimensional Brownian motion. Our aim is to identify a sufficient criterion for the uniform Laplace principle (ULP) concerning the family $\{\mathrm{X}_{\e,u^0}\}_{\e>0}$.

	\begin{condition}\label{Cond2}
		There exists a measurable map $\mathscr{G}_0:\EE_0\times \C([0,T]; \R^{\infty})\to\EE$ such that the following hold:
		\begin{enumerate}
			\item For $N\in\N$ and compact subset $\mathrm{K}\subset\EE_0$, let $\mathfrak{h}_m,\mathfrak{h}\in\UU^N$ and $u^0_m, u^0\in \EE_0$ be such that $\mathfrak{h}_m\xrightarrow{w}\mathfrak{h}$ and $u^0_m\to u^0$, as $m\to\infty$. Then 
			the set $$\mathrm{K}_N=\bigg\{\mathscr{G}_0\bigg(u^0,\int_0^{\cdot}\dot{\mathfrak{h}}{(s)}\d s\bigg): \dot{\mathfrak{h}}\in\UU^N ~\text{and}~ u^0\in \mathrm{K}\bigg\}$$ is a compact subset of $\EE$, that is,
			\begin{align*}
				\mathscr{G}_0\bigg(u^0_m,\int_0^{\cdot}\dot{\mathfrak{h}}^m(s)\d s\bigg)\to 	\mathscr{G}_0\bigg(u^0,\int_0^{\cdot}\dot{\mathfrak{h}}(s)\d s\bigg), \ \ \text{ as } \ \ m\to\infty.
			\end{align*}
			
			\item For $N\in\N$, let $\mathfrak{h}^\e,\mathfrak{h}\in\PP_2^N$  be such that $\mathfrak{h}^\e$ converges in distribution (as $\UU^N$-valued random elements) to $\mathfrak{h}$, and the initial data $u^0_\e\to u^0$, as $\e\to0$ in $\EE_0$. Then 
			\begin{align*}
				\mathscr{G}_\e\bigg(u^0_\e,\sqrt{\e}\boldsymbol{\beta}+\int_0^{\cdot}\dot{\mathfrak{h}}_\e{(s)}\d s\bigg)\Rightarrow	\mathscr{G}_0\bigg(u^0,\int_0^{\cdot}\mathfrak{h}(s)\d s\bigg),   \ \ \text{ as } \ \  \e\to 0,
			\end{align*}where $``\Rightarrow"$ represents the convergence in distribution.
		\end{enumerate}
		
		For $\psi\in\EE$ and $u^0\in\EE_0$, we define 
		\begin{align*}
			\SS_\psi:=\bigg\{\mathfrak{h}\in\mathcal{H}: \ \psi=\mathscr{G}_0\bigg(u^0,\int_0^{\cdot}\dot{\mathfrak{h}}(s)\d s\bigg)\bigg\}, 
		\end{align*}and the rate function as 
		\begin{equation}\label{RF}
			\I_{u^0}(\psi)=
			\left\{
			\begin{aligned}
				&	\inf_{\mathfrak{h}\in\SS_{\psi}}\bigg\{\frac{1}{2}\int_0^T\|\dot{\mathfrak{h}}(s)\|_{\ell^2}^2\d s\bigg\}, \ \ &&\text{ if } \ \ \psi\in\EE,\\
				&+\infty, \ \ &&\text{ if } \ \ \SS_\psi=\emptyset,
			\end{aligned}
			\right.
		\end{equation}where \begin{align*}
			\inf_{\mathfrak{h}\in\SS_{\psi}}\bigg\{\frac{1}{2}\int_0^T\|\dot{\mathfrak{h}}(s)\|_{\ell^2}^2\d s\bigg\}&:=\inf_{\mathfrak{h}\in\SS_{\psi}}\bigg\{\frac{1}{2}\sum_{j\in\N}\int_0^T|\dot{\mathfrak{h}}_j(s)|^2\d s\bigg\}\\&\ =\inf_{\left\{\mathfrak{h}\in\mathcal{H}:~   \psi=\mathscr{G}_0(u^0,\int_0^{\cdot}\dot{\mathfrak{h}}(s)\d s)\right\}}\bigg\{\frac{1}{2}\int_0^T\|\dot{\mathfrak{h}}(s)\|_{\ell^2}^2\d s\bigg\}.
		\end{align*}	
	\end{condition}
	The following theorem provides the general criteria for an LDP   for a family $\{\mathrm{X}_{\e,u^0}\}_{\e>0}$.
	
	\begin{theorem}[{\cite[Theorem 7]{ADPDVM}}]\label{thrmLDP}
		Let $\mathscr{G}_0:\EE_0\times \C([0,T];\R^\infty)\to\EE$ be a measurable map, and the Condition \ref{Cond2} holds. Then, the family $\{\mathrm{X}_{\e,u^0}\}_{\e>0}$ satisfies the Laplace principle on $\EE$ with the rate function $\I_{u^0}(\cdot)$ given in \eqref{RF}, uniformly over the initial data $u^0$ on the compact subsets of $\EE_0$.
	\end{theorem}

	Returning  to our model, let us introduce a process $\mathrm{Z}_\e:= \frac{u_\e-u_0}{\sqrt{\e}\lambda(\e)}$, $\e>0$. In view of the integral equations \eqref{mild1} and \eqref{1.4}, the process  $\mathrm{Z}_\e(\cdot)$ satisfies the following integral equation for all $t\in[0,T]$, $\P$-a.s.,
	\begin{align}\label{4.2}\nonumber
		&\mathrm{Z}_\e(t,x)
		\\&:=\frac{\alpha}{\delta+1} \frac{1}{\sqrt{\e} \lambda(\e)}\int_0^t\int_0^1\frac{\partial G}{\partial y}(t-s,x,y)\left(\p\big(u_0(s,y)+\sqrt{\e} \lambda(\e)	\mathrm{Z}_\e(s,y)\big)-\p(u_0(s,y))\right)\d y\d s\nonumber \\&\nonumber\quad + \frac{\beta}{\sqrt{\e} \lambda(\e)}\int_0^t\int_0^1G(t-s,x,y)\left(\c\big(u_0(s,y)+\sqrt{\e} \lambda(\e)	\mathrm{Z}_\e(s,y)\big)-\c(u_0(s,y))\right)\big)\d y\d s\\
		&\quad+ \frac{1}{ \lambda(\e)}\sum_{j\in\N}
		\int_0^t\int_0^1G(t-s,x,y)g\big(s,y,u_0(s,y)+\sqrt{\e} \lambda(\e)	\mathrm{Z}_\e(s,y)\big)q_j\varphi_j(y)\d \beta_j(s)\d y \nonumber	\\
		&=: I_1(t,x)+I_2(t,x)+I_3(t,x).
	\end{align}

	\begin{theorem}\label{thrm4.6}
		Assume that Hypothesis \ref{hyp1} holds and the initial data $u^0\in\L^p([0,1])$, for $p>\max\{6,2\delta+1\}$. Then, there exists a unique $\L^p([0,1])$-valued $\mathscr{F}_t$-adapted continuous process $\mathrm{Z}_\e(\cdot)$    satisfying the integral equation \eqref{4.2} such that 
		\begin{align}\label{43}
			\E\bigg[\sup_{t\in[0,T\wedge\tau^{M,\e}]}\|\mathrm{Z}_\e(t)\|_{\L^p}^p\bigg]\leq C(\alpha,\beta,\gamma,\delta,p,T,M,K,u_0),
		\end{align}
		where $\tau^{M,\e}$ is the stopping time defined 
		in  \eqref{stoppingtime}.  Further, for any $\rho>0,$ we have
		\begin{align}\label{4.4}
			\lim_{\rho\to\infty}\sup_{\e\in(0,1]}	\P\bigg(\sup_{t\in[0,T]}\|\mathrm{Z}_{\e}(t)\|_{\L^p}>\rho\bigg)=0.
		\end{align}
	\end{theorem}
	\begin{proof}
		Existence and uniqueness results for $\mathrm{Z}_\e(\cdot)$ can be obtained by combining  the proofs of Theorems \ref{thrmE1} and \ref{thrmE2}. We are providing a proof of the estimates \eqref{43} and \eqref{4.4} only.
		
		\vspace{2mm} 
		\noindent
		\emph{Proof of \eqref{43}.} From \eqref{4.2}, we have
		\begin{align}\label{I14.4}
			\|I_1(t)\|_{\L^p}\leq C \frac{\alpha}{\delta+1} \frac{1}{\sqrt{\e} \lambda(\e)}\int_0^t (t-s)^{-1+\frac{1}{2p}}\left\|\left(\p\big(u_0(s)+\sqrt{\e} \lambda(\e)	\mathrm{Z}_\e(s)\big)-\p(u_0(s))\right)\right\|_{\L^1}\d s.
		\end{align}
		By Taylor's formula, there exists a $\vartheta^\e \in (0,1)$ such that 
		\begin{align}\label{I14.5}
			|\p\big(u_0+\sqrt{\e} \lambda(\e)	\mathrm{Z}_\e\big)-\p(u_0)|&=\sqrt{\e} \lambda(\e)| \p'(u_0+\vartheta^\e\sqrt{\e} \lambda(\e)\mathrm{Z}_\e) \Z_\e|\nonumber\\
			&= (\delta+1)\sqrt{\e} \lambda(\e)|(u_0+\vartheta^\e(u_\e-u_0))^\delta \Z_\e|
			\nonumber\\& \leq 2^{\delta-1}(\delta+1)\sqrt{\e} \lambda(\e)((1-\vartheta^\e)^\delta |u_0|^\delta |\Z_\e|+\vartheta^\delta |u_\e^\delta|| \Z_\e|)
			\nonumber\\
			&\leq 2^{\delta-1}(\delta+1)\sqrt{\e} \lambda(\e)\left(|u_0|^\delta|\Z_\e|+|u_\e|^\delta|\Z_\e|\right).
		\end{align}	
		For the stopping time $\tau^{M,\e}$ defined  in  \eqref{stoppingtime}, let us consider \eqref{I14.4}, raise to  the $p^{\mathrm{th}}$ power, take supremum from $0$ to $T\wedge \tau^{M,\e}$ and then take expectation on both sides and employ \eqref{I14.5} and H\"older's inequality, to find
		\begin{align*}
			&\E\bigg[\sup_{t\in[T\wedge\tau^{M,\e}]}\|I_1(t)\|_{\L^p}^p\bigg]
			\nonumber\\
			&\leq C\E\bigg[\bigg(\int_0^{T\wedge\tau^{M,\e}} (t-s)^{-1+\frac{1}{2p}}\left\{\|u_\e(s)\|_{\L^{\frac{p\delta}{p-1}}}^\delta\|\mathrm{Z}_\e(s)\|_{\L^p}+\|u_0(s)\|_{\L^{\frac{p\delta}{p-1}}}^\delta\|\mathrm{Z}_\e(s)\|_{\L^p}
			\right\}\d s\bigg)^p\bigg]\nonumber\\
			&
			\leq C \E\bigg[\int_0^{T\wedge\tau^{M,\e}}(t-s)^{-1+\frac{1}{2p}}\|u_\e(s)\|_{\L^{\frac{p\delta}{p-1}}}^{p\delta}\|\Z_\e(s)\|_{\L^p}^p\d s\bigg]
			\nonumber\\
			&\quad	+C \bigg(\int_0^T (t-s)^{-1+\frac{1}{2p}}\|u_0(s)\|_{\L^{\frac{p\delta}{p-1}}}^{\frac{p\delta}{p-1}}\d s \bigg)^{p-1} \E\bigg[\int_0^{T\wedge\tau^{M,\e}}(t-s)^{-1+\frac{1}{2p}}\|\Z_\e(s)\|_{\L^p}^p\d s\bigg].
		\end{align*}
		Using the stopping time defined in \eqref{stoppingtime} and the estimate \eqref{E31},  we obtain
		\begin{align}\label{thrm4.71}
			&\E\bigg[\sup_{t\in[0,T\wedge\tau^{M,\e}]}\|I_1(t)\|_{\L^p}^p\bigg]\leq  C(\alpha,\delta,T,M,u_0)\E\bigg[\int_0^{T\wedge\tau^{M,\e}}(t-s)^{-1+\frac{1}{2p}}\|\Z_\e(s)\|_{\L^p}^p\d s\bigg].
		\end{align}
		For the term $I_2$ appearing in the right hand side of \eqref{4.2}, we have
		\begin{align*}
			\|I_2(t)\|_{\L^p}&\leq C  \frac{\beta}{\sqrt{\e} \lambda(\e)}\int_0^t (t-s)^{-\frac12+\frac{1}{2p}}\left\|\left(\c\big(u_0(s)+\sqrt{\e} \lambda(\e)	\mathrm{Z}_\e(s)\big)-\c(u_0(s))\right)\right\|_{\L^1}\d s\nonumber\\
			&=C  \frac{\beta}{\sqrt{\e} \lambda(\e)}\int_0^t (t-s)^{-\frac12+\frac{1}{2p}}\left\|(1+\gamma)\left(\big(u_0(s)+\sqrt{\e} \lambda(\e)	\mathrm{Z}_\e(s)\big)^{\delta+1}-(u_0(s))^{\delta+1}\right)\right.
			\nonumber\\
			&\left. \qquad -\gamma \sqrt{\e} \lambda(\e)	\mathrm{Z}_\e(s) -\left(\big(u_0(s)+\sqrt{\e} \lambda(\e)	\mathrm{Z}_\e(s)\big)^{\delta+2}-(u_0(s))^{\delta+2}\right) \right\|_{\L^1}\d s.
		\end{align*}
		By performing similar calculations as in the estimate for the term $I_1$ (see \eqref{I14.4}-\eqref{thrm4.71}), one can deduce
		\begin{align}\label{thrm4.72}	\E\bigg[\sup\limits_{t\in[0,T\wedge\tau^{M,\e}]}\|I_2(t)\|_{\L^p}^p\bigg]\leq C(\beta,\gamma,\delta,T,M,u_0)\E\bigg[\int_0^{T\wedge\tau^{M,\e}}(t-s)^{-\frac12+\frac{1}{2p}}\|\Z_\e(s)\|_{\L^p}^p\d s\bigg].
		\end{align}
				From \eqref{noiseEstimate} (cf. \cite[Eq. (2.9)]{KKM}), we have
				\begin{align}\label{thrm4.73}
					\E\bigg[\sup_{t\in[0,T\wedge \tau^{M,\e}]}\|I_3(t)\|_{\L^p}^p\bigg] \leq C(p,K,M,T).
				\end{align}Combining \eqref{thrm4.71}, \eqref{thrm4.72} and \eqref{thrm4.73} with \eqref{4.2} and applying Gronwall's inequality, the required result \eqref{43} follows.
			\end{proof}
			\begin{proof}[Proof of \eqref{4.4}]
				Further, for the stopping time defined in \eqref{stoppingtime}, by using Markov's inequality and Lemma \ref{lem4.1.1}, we have
				\begin{align*}
					\sup_{\e\in (0,1]}\P\bigg(\sup_{t\in[0,T]}\|\mathrm{Z}_{\e}(t)\|_{\L^p}>\rho\bigg)&\leq  \sup_{\e\in (0,1]}\P\bigg(\sup_{t\in[0,T\wedge \tau^{M,\e} ]}\|\mathrm{Z}_{\e}(t)\|_{\L^p}>\rho\bigg) +\sup_{\e\in (0,1]}\P\big(\tau^{M,\e}<T\big)\\
					&\leq \frac{1}{\rho^p}\sup_{\e\in (0,1]}\E\bigg[\sup_{t\in[0,T\wedge \tau^{M,\e}]}\|\Z_\e(t)\|_{\L^p}^p\bigg]+\sup_{\e\in (0,1]}\P\big(\tau^{M,\e}<T\big)\\
					&\leq \frac{1}{\rho^p} C(p, K,M, T)+\sup_{\e\in (0,1]}\P\big(\tau^{M,\e}<T\big).		
				\end{align*}
				For each fixed $M$, we obtain
				\begin{align*}
					\lim_{\rho\to \infty}\sup_{\e\in (0,1]}\P\bigg(\sup_{t\in[0,T]}\|\mathrm{Z}_{\e}(t)\|_{\L^p}>\rho\bigg)\leq\sup_{\e\in (0,1]}\P\big(\tau^{M,\e}<T\big). 
				\end{align*}
				Now, passing $M\to\infty$ in both sides on the above inequality, we find
				\begin{align*}
					\lim_{\rho\to \infty}\sup_{\e\in (0,1]}\P\bigg(\sup_{t\in[0,T]}\|\mathrm{Z}_{\e}(t)\|_{\L^p}>\rho\bigg)=0,
				\end{align*}
				which completes the proof of \eqref{4.4}.
			\end{proof}
			
			The main goal of this section is to  establish an LDP for the family $\{\mathrm{Z}_{\e}\}_{\e>0}:=\{\mathrm{Z}_\e(\cdot):\e>0\}$. Once the Condition \ref{Cond2} is verified,   Theorem \ref{thrmLDP} assures that the family $\{\mathrm{Z}_{\e}\}_{\e>0}$ satisfies a ULP.
			Subsequently, by Theorem \ref{MDPthm} (see \cite{ADPDVM}), we deduce that an LDP holds for the family  $\{\mathrm{Z}_{\e}\}_{\e>0}$. Therefore, our aim is to  verify the Condition \ref{Cond2} for the family $\{\mathrm{Z}_\e\}_{\e>0}$  defined in \eqref{4.2}. To accomplish this, we  introduce controlled and skeleton integral equations associated with \eqref{4.2} in the following subsection:

			\subsection{Controlled and skeleton equations}
			Let us fix $\EE:=\C([0,T];\L^p([0,1]))$,  $\EE_0:=\L^p([0,1]),$ for $p>\max\{6,2\delta+1\}$.   By an application of the celebrated Yamada-Watanabe theorem (for details see \cite[Definitions 1.8 and 1.9, Theorem 2,1]{MRBSXZ}),  a unique mild solution coincides with a probabilistically strong solution (\cite[Proposition 3.7]{IGCR}), and the solution mapping for the integral equation \eqref{4.2} can be defined as $$\mathrm{Z}_\e(\cdot):=\mathscr{G}_\e\big(u^0,\sqrt{\e}\boldsymbol{\beta}(\cdot)\big).$$ Let us now consider  the following stochastic controlled integral (SCI) equation for all $t\in[0,T]$, $\P$-a.s.:
			\begin{align}\label{CE1}
				\nonumber
				&\mathrm{Z}_{\e,\mathfrak{h}^\e}(t,x)\\&\nonumber=\frac{\alpha}{\delta+1}\frac{1}{\sqrt{\e} \lambda(\e)}\int_0^t\int_0^1\frac{\partial G}{\partial y}(t-s,x,y)\big(\p\big(u_0(s,y)+\sqrt{\e} \lambda(\e)\mathrm{Z}_{\e,\mathfrak{h}^\e}(s,y)\big)-\p(u_0(s,y))\big)\d y\d s\\&\nonumber\quad +\frac{\beta}{\sqrt{\e} \lambda(\e)}\int_0^t\int_0^1G(t-s,x,y)\big(\c\big(u_0(s,y)+\sqrt{\e} \lambda(\e)	\mathrm{Z}_{\e,\mathfrak{h}^\e}(s,y)\big)-\c(u_0(s,y))\big)\d y\d s\\&\nonumber\quad+\frac{1}{ \lambda(\e)} \sum_{j\in\N}
				\int_0^t\int_0^1G(t-s,x,y)\big(\sigma_j\big(s,y,(u_0+\sqrt{\e} \lambda(\e)	\mathrm{Z}_{\e,\mathfrak{h}^\e})(s,y)\big)\big)\d y \d \beta_j(s)\\
				&\quad+ \sum_{j\in\N}
				\int_0^t\int_0^1G(t-s,x,y)\sigma_j\big(s,y,(u_0+\sqrt{\e} \lambda(\e)	\mathrm{Z}_{\e,\mathfrak{h}^\e})(s,y)\big)\dot{\mathfrak{h}}_j^\e(s)\d y\d s,
			\end{align}
			where $\sigma_j(\cdot,\cdot,u(\cdot,\cdot)):=q_jg(\cdot,\cdot,u(\cdot,\cdot)\varphi_j(\cdot)$.
			
			By an application of Girsanov's theorem (see \cite[Theorem 10.14]{DaZ}),    the solution $\mathrm{Z}_{\e,\mathfrak{h}^\e}(\cdot)$ of the integral equation \eqref{CE1} can be represented as  $$\mathrm{Z}_{\e,\mathfrak{h}^\e}(\cdot):=\mathscr{G}_{\e}\bigg(u^0, \boldsymbol{\beta}(\cdot)+\lambda(\e)\int_0^{\cdot}\dot{\mathfrak{h}}^{\e}(s)\d s\bigg).$$
			For more details, see Proposition  \ref{prop4.6} below.
			
			It is worth noticing that  for every $u\in\L^2([0,1])$, the linear map $\sigma(\cdot,\cdot,u):\big(\sigma_j(\cdot,\cdot,u)\big)_{j\in\N}= \big(q_jg(\cdot,\cdot,u)\varphi_j\big)_{j\in\N}:\ell^2\to \L^2([0,1])$ defined by 
			\begin{align}\label{CE051}
				\sigma(\cdot,\cdot,u)\kappa:=\sum_{j\in\N}\sigma_j(\cdot,\cdot,u)\mathrm{\kappa}_j=\sum_{j\in\N}q_jg(\cdot,\cdot,u)\varphi_j\mathrm{\kappa}_j, \ \ \text{ for } \ \ \mathrm{\kappa}=(\mathrm{\kappa}_j)_{j\in\N}\in\ell^2,
			\end{align}is in $\L_2(\ell^2,\L^2([0,1]))$ (the space of Hilbert-Schmidt operators from $\ell^2$ to $\L^2([0,1]))$. For a proof, one may refer to \cite[pp. 27]{KKM}.

			Let us denote the solution  $\mathrm{Z}_{0,\mathfrak{h}}(\cdot):=\mathrm{Z}_\mathfrak{h}(\cdot):=\mathscr{G}_0\big(\int_0^{\cdot}\dot{\mathfrak{h}}(s)\d s\big)$ of the following skeleton equation for all $t\in[0,T]$,
			\begin{align}\label{CE2}
				\nonumber
				\mathrm{Z}_\mathfrak{h}(t,x)&=\frac{\alpha}{\delta+1}\int_0^t\int_0^1\frac{\partial G}{\partial y}(t-s,x,y)\p'(u_{0}(s,y))\mathrm{Z}_\mathfrak{h}(s,y)\d y\d s\\&\nonumber\quad +\beta\int_0^t\int_0^1G(t-s,x,y)\c'(u_{0}(s,y))\mathrm{Z}_\mathfrak{h}(s,y)\d y\d s\\&\quad+ \sum_{j\in\N}
				\int_0^t\int_0^1G(t-s,x,y)\sigma_j(s,y,u_0(s,y))\dot{\mathfrak{h}}_j(s)\d y\d s.
			\end{align}For $\psi\in \C([0,T];\L^p([0,1]))$, we define the rate function
			\begin{align}\label{CE3}
				\I_{u^0}(\psi ):=\inf_{\big\{\mathfrak{h}\in \mathcal{H}:\  \psi=\mathscr{G}_0(u^0,\int_0^{\cdot}\dot{\mathfrak{h}}(s)\d s)\big\}}\bigg\{\frac{1}{2} \int_0^T\|\dot{\mathfrak{h}}(s)\|^2_{\ell_2}\d s\bigg\}.
			\end{align}
			The following results provide the existence and uniqueness of solutions to the systems \eqref{CE1} and \eqref{CE2}, respectively.
			\begin{proposition}\label{prop4.6}
				Assume that Hypothesis \ref{hyp1} holds and $\mathfrak{h}\in\PP_2^N$. Then, for any initial data $u^0\in\L^p([0,1])$,  for $p>\max\{6,2\delta+1\}$, there exists a unique $\L^p([0,1])$-valued $\mathscr{F}_t$-adapted continuous process $\mathrm{Z}_{\e,\mathfrak{h}^\e}(\cdot)$ satisfying the SCI equation \eqref{CE1} such that
				\begin{align}\label{4p8}
					\sup_{\e\in(0,1]}	\E\bigg[\sup_{t\in[0,T\wedge \tau^{M,\e}]}\|\mathrm{Z}_{\e,\mathfrak{h}^\e}(t)\|_{\L^p}^p\bigg]\leq C<\infty.
				\end{align}
				Further, for any $\rho>0,$ we have
				\begin{align}\label{4p8.1}
					\lim_{\rho\to\infty}\sup_{\e\in(0,1]}	\P\bigg(\sup_{t\in[0,T]}\|\mathrm{Z}_{\e,\mathfrak{h}^\e}(t)\|_{\L^p}>\rho\bigg)=0.
				\end{align}
			\end{proposition}
			
			\begin{proof}
				The proof of this proposition can be obtained with the help of Girsanov's theorem (\cite[Theorem 10.14]{DaZ}) in the following way:
				
				
				For any fixed $\mathfrak{h} \in\PP_2^N$, we define 
				\begin{align*}
					\frac{\d \vi{\P}}{\d \P}:=\exp\bigg\{-\lambda(\e)\int_0^T\dot{\mathfrak{h}}(s) \d\boldsymbol{\beta}(s)-\frac{\lambda^2(\e)}{2}\int_0^T\|\dot{\mathfrak{h}}(s)\|_{\ell^2}^2\d s\bigg\},
				\end{align*}
				where
				the stochastic integral $\displaystyle\int_0^t\dot{\mathfrak{h}}(s)\d\boldsymbol{\beta}( s)=\sum\limits_{j\in\mathbb{N}}\int_0^t\dot{\mathfrak{h}}_j(s)\d\beta_j(s)$, for all $t\in[0,T]$. 
				Using Novikov's condition, one can show that 
				\begin{align*}
					\exp\bigg\{-\lambda(\e)\int_0^T\dot{\mathfrak{h}}(s)\d\boldsymbol{\beta}( s)-\frac{\lambda^2(\e)}{2}\int_0^T{\|\dot{\mathfrak{h}}(s)\|_{\ell^2}^2}\d s\bigg\},
				\end{align*}is an exponential martingale (Dol\'eans-Dade exponential), $\vi{\P}$ is an another probability measure on the probability space $(\Omega,\mathscr{F},\{\mathscr{F}_t\}_{t\geq 0},\P)$.  The measure $\vi{\P}$ is equivalent to the probability measure $\P$.
				
				Define a process 
				\begin{align}\label{GT}
					\vi{\boldsymbol{\beta}}(t):= \boldsymbol{\beta}(t)+\lambda(
					\e)\int_0^t\dot{\mathfrak{h}}(s)\d s, \ \ t\in[0,T]. 
				\end{align}
				Then, with the help of Girsanov's theorem, we determine that $\vi{\boldsymbol{\beta}}(\cdot)$ becomes an infinite-dimensional Wiener process under the new probability measure $\vi{\P}$. Due to the  equivalency of  the measures $\vi{\P}$ and $\P$, we find that $\vi{\boldsymbol{\beta}}(\cdot)$ is also an infinite-dimensional  Wiener process with respect to ${\P}$. On the probability space $(\Omega,\mathscr{F},\{\mathscr{F}_t\}_{t\geq 0},\vi{\P})$,  Theorem \ref{thrm4.6} ensures both the existence and uniqueness of the solution of the SCI equation \eqref{CE1} under the newly constructed probability measure $\vi{\P}$. Consequently, this also confirms the well-posedness of the  SCI equation \eqref{CE1} with respect to the probability measure $\P$.

				Next, our aim is to show the estimate \eqref{4p8}. In view of the estimate \eqref{43}, it is sufficient to prove the estimate for the final term appearing in the right-hand side of \eqref{CE1}, which can be handled in the following way: 
				\begin{align}\label{4p21.1}\nonumber
					&\E\bigg[\sup_{t\in[0,T\wedge \tau^{M,\e}]}\bigg\|\sum_{j\in\N}
					\int_0^t\int_0^1G(t-s,\cdot,y)\sigma_j\big(s,y,(u_0+\sqrt{\e} \lambda(\e)	\mathrm{Z}_{\e,\mathfrak{h}^\e})(s,y)\big)\dot{\mathfrak{h}}_j^\e(s)\d y\d s\bigg\|_{\L^p}^p\bigg]\\&\nonumber\leq 
					\E\bigg[\bigg(\sum_{j\in\N}\int_0^{T\wedge \tau^{M,\e}}(t-s)^{-\frac{1}{2}+\frac{1}{2p}}\|g(s,(u_0+\sqrt{\e} \lambda(\e)	\mathrm{Z}_{\e,\mathfrak{h}^\e})(s))\|_{\L^2}\|\varphi_j\|_{\L^2}q_j |\dot{\mathfrak{h}}_j^\e(s)|\d s\bigg)^p\bigg]\\\nonumber& 
					\leq 
					C\E\bigg[\bigg\{\bigg(	\sum_{j\in\N}\int_0^{T\wedge \tau^{M,\e}}(t-s)^{-1+\frac{1}{p}}\|g(s,(u_0+\sqrt{\e} \lambda(\e)	\mathrm{Z}_{\e,\mathfrak{h}^\e})(s))\|_{\L^2}^2\|\varphi_j\|_{\L^2}^2q_j^2\d s\bigg)^{\frac{1}{2}}\\&\qquad\times\nonumber\bigg(\sum_{j\in\N}\int_0^t(\dot{\mathfrak{h}}_j^\e)^2(s)\d s\bigg)^{\frac{1}{2}}\bigg\}^p\bigg]
					\\&\nonumber=C
					\E\bigg[\bigg\{\bigg(\sum_{j\in\N}q_j^2\bigg)^{\frac{1}{2}}\left(\int_0^{T\wedge \tau^{M,\e}}(t-s)^{-1+\frac{1}{p}}\|g(s,(u_0+\sqrt{\e} \lambda(\e)	\mathrm{Z}_{\e,\mathfrak{h}^\e})(s))\|_{\L^2}^2\d s\right)^{1/2}\\&\nonumber\qquad\times\left(\int_0^t\|\dot{\mathfrak{h}}^\e(s)\|_{\ell^2}^2\d s\right)^{1/2}\bigg\}^p\bigg] \\&\nonumber \leq C(N)	\E\bigg[\bigg(\int_0^{T\wedge \tau^{M,\e}}(t-s)^{-1+\frac{1}{p}}\|g(s,(u_0+\sqrt{\e} \lambda(\e)	\mathrm{Z}_{\e,\mathfrak{h}^\e})(s))\|_{\L^2}^2\d s\bigg)^{p/2}\bigg]\\&\nonumber\leq C(p,K,N,T)\E\bigg[\int_0^{T\wedge \tau^{M,\e}}(t-s)^{-1+\frac{1}{p}}\big(1+\|u_0+\sqrt{\e} \lambda(\e)	\mathrm{Z}_{\e,\mathfrak{h}^\e}(s))\|_{\L^p}^p\big)\d s\bigg]\\& \leq C(p,K,N,T)\bigg\{1+\sup_{s\in[0,T]}\|u_0(s)\|_{\L^p}^p+(\sqrt{\e}\lambda(\e))^{p}\E\bigg[\int_0^{T\wedge \tau^{M,\e}}\|\Z_{\e,\mathfrak{h}^\e}(s)\|_{\L^p}^p\d s\bigg]\bigg\},
				\end{align}where we have used the fact that $\sum\limits_{j\in\N}q_j^2<\infty$ and $\{\varphi_j\}_{j\in\N}$ is an orthonormal basis. The final estimate \eqref{4p8} can be obtained with the help of \eqref{thrm4.71}, \eqref{thrm4.72}, \eqref{thrm4.73} and \eqref{4p21.1} followed by an application of Gronwall's inequality.
				
				We are not providing a proof of \eqref{4p8.1} as it will follow on the same lines as in the proof of \eqref{4.4}.
			\end{proof}

			\begin{proposition}\label{prop4.7}
				Assume that Hypothesis \ref{hyp1} holds, and {$\mathfrak{h}\in \mathcal{H}$}. Then, for any initial data $u^0\in\L^p([0,1])$, for $p\geq 2\delta+1$ , there exists a unique solution $\mathrm{Z}_\mathfrak{h}\in \C([0,T];\L^p([0,1]))$ to the skeleton integral equation \eqref{CE2} such that 
				\begin{align}\label{4p9}
					\sup_{t\in[0,T]}\|\mathrm{Z}_\mathfrak{h}(t)\|_{\L^p}^p\leq C<\infty.
				\end{align}
			\end{proposition}
			\begin{proof}
				The proof of this proposition is similar to the proof of  \cite[Propositions 3.3 and 3.5]{AKMTM3}. In order to complete the proof of this proposition, we only need to establish the estimate \eqref{4p9}. 
				
				\vspace{2mm}
				\noindent
				\emph{Proof of \eqref{4p9}.} Taking $\L^p$-norm on both sides of  \eqref{CE2}, we find for all $t\in[0,T]$
				\begin{align}\label{prop4.90}
					\nonumber
					\|	\mathrm{Z}_\mathfrak{h}(t)\|_{\L^p}&=\frac{\alpha}{\delta+1}\bigg\|\int_0^t\int_0^1\frac{\partial G}{\partial y}(t-s,\cdot,y)\p'(u_{0}(s,y))\mathrm{Z}_\mathfrak{h}(s,y)\d y\d s\bigg\|_{\L^p}\\&\nonumber\quad +\beta\bigg\|\int_0^t\int_0^1G(t-s,\cdot,y)\c'(u_{0}(s,y))\mathrm{Z}_\mathfrak{h}(s,y)\d y\d s\bigg\|_{\L^p}\\&\quad+\bigg\|\sum_{j\in\N}
					\int_0^t\int_0^1G(t-s,\cdot,y)\sigma_j(s,y,u_0(s,y))\dot{\mathfrak{h}}_j(s)\d y\d s\bigg\|_{\L^p}
					\nonumber\\&=:\frac{\alpha}{\delta+1}I_1(t)+\beta I_2(t)+I_3(t).
				\end{align}Let us consider the $I_1$ and estimate it using the estimates \eqref{A2}, \eqref{A7}, Young's and H\"older's inequalities, the embedding $\L^p([0,1])\subset \L^{\frac{p}{p-\delta}}([0,1])$, for $p>\delta$,  and \eqref{E31}, as
				\begin{align}\label{prop4.91}\nonumber
					I_1(t)&\leq C\int_0^t(t-s)^{-1+\frac{1}{2p}}\|u_0^\delta(s)\Z_\mathfrak{h} (s)\|_{\L^1}\d s \\&\nonumber
					\leq C\int_0^t(t-s)^{-1+\frac{1}{2p}}\|u_0^\delta(s)\|_{\L^\frac{p}{\delta}}\|\Z_\mathfrak{h}(s)\|_{\L^\frac{p}{p-\delta}}\d s\\&\leq 
					C\sup_{s\in[0,t]}\|u_0(s)\|_{\L^p}^\delta \int_0^t(t-s)^{-1+\frac{1}{2p}}\|\Z_\mathfrak{h}(s)\|_{\L^p}\d s.
				\end{align}Let us re-write the term $I_2$ appearing in the right hand side of \eqref{prop4.90} as
				\begin{align}\label{prop4.92}\nonumber
					I_2(t)&=  \bigg\|\int_0^t\int_0^1 G(t-s,\cdot,y)\big(-\gamma+(1+\gamma)(1+\delta)u^\delta(s)-(2\delta+1)u^{2\delta}(s)\big)\Z_\mathfrak{h}(s)\d y\d s\bigg\|_{\L^p}\\&\nonumber\leq
					\gamma \bigg\|\int_0^t\int_0^1 G(t-s,\cdot,y)\Z_\mathfrak{h}(s)\d y\d s\bigg\|_{\L^p}\\&\nonumber\quad +(1+\gamma)(1+\delta)\bigg\|\int_0^t\int_0^1 G(t-s,\cdot,y)u^\delta(s)\Z_\mathfrak{h}(s)\d y\d s\bigg\|_{\L^p}\\&\nonumber\quad +(2\delta+1)\bigg\|\int_0^t\int_0^1 G(t-s,\cdot,y)u^{2\delta}(s)\Z_\mathfrak{h}(s)\d y\d s\bigg\|_{\L^p}\\&=: \gamma I_{21}(t)+(1+\gamma)(1+\delta)I_{22}(t)+(2\delta+1)I_{23}(t).
				\end{align}We are providing an estimate  for $I_{23}$ (considering the highest order non-linearity) only, rest of the terms $I_{21}$ and $I_{22}$ can be handled in a similar way. For the term $I_{23}$, we use similar arguments to \eqref{prop4.91}, and estimate it for $p\geq 2\delta+1$, as
				\begin{align}\label{prop4.93}\nonumber
					I_{23}(t)&\leq C\int_0^t(t-s)^{-\frac{1}{2}+\frac{1}{2p}}\|u_0^{2\delta}(s)\Z_{\mathfrak{h}}(s)\|_{\L^p}\d s\\&\nonumber\leq C\int_0^t (t-s)^{-\frac{1}{2}+\frac{1}{2p}}\|u_0^{2\delta}(s)\|_{\L^\frac{p}{2\delta}}\|\Z_\mathfrak{h}(s)\|_{\L^{\frac{p}{p-2\delta}}}\d s \\&\leq C \sup_{s\in[0,t]}\|u_0(s)\|_{\L^p}^{2\delta}\int_0^t(t-s)^{-\frac{1}{2}+\frac{1}{2p}}\|\Z_{\mathfrak{h}}(s)\|_{\L^p}\d s.
				\end{align}Using the fact that $\mathfrak{h}\in\mathcal{H}$, we estimate the final term $I_3$ as
				\begin{align}\label{prop4.94}\nonumber
					I_3(t)& \nonumber\leq C
					\sum_{j\in\N}\int_0^t(t-s)^{\frac{1}{2p}-\frac{1}{2}}\|g(s,u_{0}(s))\varphi_j\|_{\L^1}|q_j| |\dot{\mathfrak{h}}_j(s)|\d s
					\\&\nonumber \leq C
					\sum_{j\in\N}\int_0^t(t-s)^{\frac{1}{2p}-\frac{1}{2}}\|g(s,u_{0}(s))\|_{\L^2}\|\varphi_j\|_{\L^2}|q_j| |\dot{\mathfrak{h}}_j(s)|\d s
					\\&\nonumber \leq 
					C\bigg(	\sum_{j\in\N}\int_0^t(t-s)^{\frac{1}{p}-1}\|g(s,u_{0}(s))\|_{\L^2}^2\|\varphi_j\|_{\L^2}^2q_j^2\d s\bigg)^{\frac{1}{2}}\bigg(\sum_{j\in\N}\int_0^t\dot{\mathfrak{h}}_j^2(s)\d s\bigg)^{\frac{1}{2}}
					\\&\nonumber=C
					\bigg(\sum_{j\in\N}q_j^2\bigg)^{\frac{1}{2}}\left(\int_0^t(t-s)^{\frac{1}{p}-1}\|g(s,u_{0}(s))\|_{\L^2}^2\d s\right)^{1/2}\left(\int_0^t\|\dot{\mathfrak{h}}(s)\|_{\ell^2}^2\d s\right)^{1/2}\nonumber\\&\nonumber\leq CK\|\mathfrak{h}\|_{\mathcal{H}}\left(\int_0^t(t-s)^{\frac{1}{p}-1}\big(1+\|u_{0}(s)\|_{\L^2}\big)^2\d s\right)^{1/2}\\&\leq CKT^{\frac{1}{2p}}\|\mathfrak{h}\|_{\mathcal{H}}\sup_{s\in[0,t]}\big\{1+\|u_{0}(s)\|_{\L^2}\big\},
				\end{align}where we have used the fact that $\sum\limits_{j\in\N}q_j^2<\infty$ and $\{\varphi_j\}_{j\in\N}$ is an orthonormal basis.
				
				Substituting \eqref{prop4.91}, \eqref{prop4.92}, \eqref{prop4.93} and \eqref{prop4.94} in \eqref{prop4.90}, and an application of Gronwall's inequality in the resultant leads to the required result \eqref{4p9}.
			\end{proof}
			Now, we are in a position to state and prove our main result of this section.
			\begin{theorem}\label{MDPthm}
				Assume that  Hypothesis \ref{hyp1} holds. Then, for any $p> \max\{6, 2\delta+1\}$, the sequence of  processes $\{\mathrm{Z}_\e\}_{\e>0}$ defined in \eqref{4.2}  satisfies an LDP on $\C([0,T];\L^p([0,1]))$ with the speed $\lambda(\e)$ and the rate  function $I$ given by \eqref{CE3}.
			\end{theorem}In order to prove our main Theorem \ref{MDPthm}, it is sufficient to verify the Condition \ref{Cond2}. In the following result, we provide a proof of  verification of the Condition \ref{Cond2} (1).
			
			\begin{lemma}[Compactness]\label{compactness}
				The set $\mathrm{K}_N=\left\{\mathscr{G}_0\left(u^0,\int_0^{\cdot}\dot{\mathfrak{h}}(s)\d s\right): \mathfrak{h}\in\UU^N \right\}$ is a compact subset of $\EE.$
			\end{lemma}
			\begin{proof}
				We prove that the map $\mathscr{G}_0(\cdot)$ is continuous with respect to the weak topology and since $\UU^N$ is compact, the compactness of $\mathrm{K}_N$ follows. 
				
				\vskip .2cm
				\noindent
				\textbf{Claim:}	If  $\mathfrak{h},\mathfrak{h}_m\subset \UU^N$ are such that $\mathfrak{h}_m\xrightarrow{w}\mathfrak{h}$, as $m\to\infty$, then the following holds:
				\begin{align}\label{4p15}
					\lim_{m\to \infty}\sup_{t\in[0,T]} \|\mathrm{Z}_{\mathfrak{h}_m}(t)-\mathrm{Z}_\mathfrak{h}(t)\|_{\L^p}=0.
				\end{align}
				From \eqref{CE2}, we have that for all $(t,x)\in[0,T]\times [0,1]$,
				\begin{align}\label{DT}\nonumber
					\mathrm{Z}_{\mathfrak{h}_m}(t,x)-\mathrm{Z}_\mathfrak{h}(t,x)&=\frac{\alpha}{\delta+1}\int_0^t\int_0^1\frac{\partial G}{\partial y}(t-s,x,y)\p'(u_{0}(s,y))\big(\mathrm{Z}_{\mathfrak{h}_m}(s,y)-\mathrm{Z}_\mathfrak{h}(s,y)\big)\d y\d s\nonumber\\&\nonumber\quad+\beta\int_0^t\int_0^1G(t-s,x,y)\c'(u_{0}(s,y))\big(\mathrm{Z}_{\mathfrak{h}_m}(s,y)-\mathrm{Z}_\mathfrak{h}(s,y)\big)\d y\d s\\&\nonumber\quad+
					\sum_{j\in\N}
					\int_0^t\int_0^1G(t-s,x,y)\sigma_j(s,y,u_0(s,y))(\dot{\mathfrak{h}_j}_{m}(s)-\dot{\mathfrak{h}}_{j}(s))\d y\d s
					\\& =:I_{1}^{m}(t,x)+I_{2}^{m}(t,x)+I_{3}^{m}(t,x).
				\end{align}
				
				
				Now, we consider the term $\|I_1^m(\cdot)\|_{\L^p}$ and estimate it using a similar calculation to \eqref{prop4.91} and Theorem \ref{thrmE2} as
				\begin{align}\label{4p17}\nonumber
					\|I_{1}^{m}(t)\|_{\L^p}&\leq C(\alpha,\delta)\int_{0}^{t}(t-s)^{-1+\frac{1}{2p}}
					\big	\|\p'(u_{0}(s))\big(\mathrm{Z}_{\mathfrak{h}_m}(s)-\mathrm{Z}_\mathfrak{h}(s)\big)\|_{\L^1}\d s
					\nonumber	\\ &\nonumber
					\leq C(\alpha,\delta)\int_{0}^{t}(t-s)^{-1+\frac{1}{2p}}
					\|u_{0}^\delta(s)\|_{\L^{\frac{p}{\delta}}}\|\mathrm{Z}_{\mathfrak{h}_m}(s)-\mathrm{Z}_\mathfrak{h}(s)\|_{\L^{\frac{p}{p-\delta}}}\d s
					\\ & \leq  C(\alpha,\delta,u_0)\int_{0}^{t}(t-s)^{-1+\frac{1}{2p}}\|\mathrm{Z}_{\mathfrak{h}_m}(s)-\mathrm{Z}_\mathfrak{h}(s)\|_{\L^p}\d s.
				\end{align}Substituting  $\c'(u)=-\gamma+(1+\gamma)(1+\delta)u^{\delta}-(2\delta+1)u^{2\delta}$ in the term $I_2^m$ appearing in the right hand side of \eqref{DT}, we find
				\begin{align}\label{4p21}\nonumber
					I_{2}^m(t,x)&=\beta\bigg\{-\gamma\int_0^t\int_0^1G(t-s,x,y)\big(\mathrm{Z}_{\mathfrak{h}_m}(s,y)-\mathrm{Z}_\mathfrak{h}(s,y)\big)\d y\d s\\&\nonumber\qquad +(1+\gamma)(1+\delta)\int_0^t\int_0^1G(t-s,x,y)u^\delta(s)\d y\d s\\&\nonumber\qquad -(2\delta+1)\int_0^t\int_0^1G(t-s,x,y)u^{2\delta}(s)\big(\mathrm{Z}_{\mathfrak{h}_m}(s,y)-\mathrm{Z}_\mathfrak{h}(s,y)\big)\d y\d s\bigg\}\\& =:  I_{21}^m(t,x)+I_{22}^m(t,x)+I_{23}^m(t,x).
				\end{align}
				
				Using a similar calculation to \eqref{4p17} (cf. \eqref{prop4.93}), we estimate the terms $\|I_{21}^m(\cdot)\|_{\L^p}$, $\|I_{22}^m(\cdot)\|_{\L^p}$ and $\|I_{23}^m(\cdot)\|_{\L^p}$ as
				\begin{align}
					\|I_{21}^{m}(t)\|_{\L^p}&\leq C(\beta,\gamma)\int_0^t \|\Z_{\mathfrak{h}_m}(s)-\Z_\mathfrak{h}(s)\|_{\L^p}\d s,\\
					\|I_{22}^{m}(t)\|_{\L^p}&\leq  C(\beta,\delta,u_0)\int_{0}^{t}(t-s)^{-\frac{1}{2}+\frac{1}{2p}}\|\mathrm{Z}_{\mathfrak{h}_m}(s)-\mathrm{Z}_\mathfrak{h}(s)\|_{\L^p}\d s, \\
					\|I_{23}^{m}(t)\|_{\L^p}&\leq  C(\beta,\delta,u_0)\int_{0}^{t}(t-s)^{-\frac{1}{2}+\frac{1}{2p}}\|\mathrm{Z}_{\mathfrak{h}_m}(s)-\mathrm{Z}_\mathfrak{h}(s)\|_{\L^p}\d s,
				\end{align}respectively.
				
				Making use of a similar calculation to \eqref{I_1}, we estimate the term $	\|I_{3}^m(\cdot)\|_{\L^p}$  as
				\begin{align}\nonumber 
					\|I_{3}^m(t)\|_{\L^p}&\nonumber\leq C(p) \sum_{j\in\N} \int_0^t(t-s)^{-\frac{1}{2}}\big\|e^{-\frac{|\cdot|^2}{a_1(t-s)}}*\sigma_j(s,\cdot,u_{0}(s,\cdot))\big(\dot{\mathfrak{h}_j}_m(s)-\dot{\mathfrak{h}}_j(s)\big)\big\|_{\L^p}\d s
					\\&\nonumber\leq C(p)\sum_{j\in\N}\int_0^t(t-s)^{-\frac{1}{4}+\frac{1}{2p}}\big\|\sigma_j(s,\cdot,u_{0}(s,\cdot))\big(\dot{\mathfrak{h}_j}_m(s)-\dot{\mathfrak{h}}_j(s)\big)\big\|_{\L^2}\d s 
					\\&\label{4.10}\leq C(p)T^{\frac{1}{4}+\frac{1}{2p}}\bigg(\sum_{j\in\N}\int_0^t\big\|\sigma_j(s,\cdot,u_{0}(s,\cdot))\big(\dot{\mathfrak{h}_j}_m(s)-\dot{\mathfrak{h}}_j(s)\big)\big\|_{\L^2}^2\d s \bigg)^\frac{1}{2}
					\\&\label{4.11}
					\to 0, \ \ \text{ as } \ \ m\to\infty,
				\end{align}
				where we have used the fact that the Hilbert-Schmidt operators on separable Hilbert spaces are compact operators.  Recall that a compact operator maps weakly convergent sequences to strongly convergent sequences. Since the sequence $\{\mathfrak{h}_m\}_{m\in\N}$ converges weakly to $\mathfrak{h}$ in $\mathfrak{U}^M$, we infer that for any $j\in\N$
				\begin{align*}
					\int_0^T\big\|\sigma_j(s,\cdot,u_{0}(s,\cdot))\big(\dot{\mathfrak{h}_j}_m(s)-\dot{\mathfrak{h}}_j(s)\big)\big\|_{\L^2}^2\d s\to 0, \  \text{ as } \  m\to\infty.
				\end{align*}
				


				Combining the above estimates in \eqref{DT}, we obtain  
				\begin{align}\label{DT2}\nonumber
					&	\sup_{t\in[0,T]}\|\mathrm{Z}_{\mathfrak{h}_m}(t)-\mathrm{Z}_\mathfrak{h}(t)\|_{\L^p}
					\\&\nonumber	\leq C(p, \delta, \alpha,\beta,\gamma,u_0)\int_{0}^{t}\left\{(t-s)^{-1+\frac{1}{2p}}+(t-s)^{-\frac{1}{2}+\frac{1}{2p}}+1\right\}\|\mathrm{Z}_{\mathfrak{h}_m}(s)-\mathrm{Z}_\mathfrak{h}(s)\|_{\L^p}\d s\\&\quad +\sup_{t\in[0,T]}\|I_{3}^{m}(t)\|_{\L^p}.
				\end{align}
				Applying Gronwall's inequality in \eqref{DT2} and then employing \eqref{4.11}, we conclude that
				\begin{align*}
					\sup_{t\in[0,T]}\|\mathrm{Z}_{\mathfrak{h}_m}(t)-\mathrm{Z}_\mathfrak{h}(t)\|_{\L^p}\leq C(p,T, \delta, \alpha,\beta,\gamma,u_0)\sup_{t\in[0,T]}\|I_{3}^{m}(t)\|_{\L^p}\to 0,  \ \text{as}\ m\to 0,
				\end{align*}
				which is the required claim \eqref{4p15}, and hence the proof.
			\end{proof}

			Let us now move to  verification of the Condition \ref{Cond2} (2).
			\begin{lemma}[Weak Convergence]\label{WC}
				Let $\mathfrak{h}^\e, \mathfrak{h}\in \mathscr{P}^2_N$ be such that $\mathfrak{h}^\e$ converges in distribution  to $\mathfrak{h}$, as $\e\to 0$. Then, $\mathrm{Z}_{\e,\mathfrak{h}^\e}(\cdot)$ converges to $\mathrm{Z}_\mathfrak{h}(\cdot)$ in distribution in the space $\C([0,T]; \L^p([0,1]))$, for $p>\max\{6,2\delta+1\}$.
			\end{lemma}
			\begin{proof}
				Since $\mathfrak{h}^\e$ converges in distribution to $\mathfrak{h}$ as $\e \to 0$, the well-known Skorokhod representation theorem can be applied. This allows us to consider the existence of a stochastic  basis $(\widetilde{\Omega}, \widetilde{\mathscr{F}},\{\widetilde{\mathscr{F}}_t\}_{t\geq 0},\widetilde{\P}),$
				a Brownian motion $\widetilde{\boldsymbol{\beta}}(\cdot)$, families of $\widetilde{\mathscr{F}}_t$-adapted predictable processes $\{\widetilde{\mathfrak{h}}^\e, \e>0\},$ and $ \widetilde{\mathfrak{h}} \in \mathscr{P}_2^N$ such that the joint law of $(\widetilde{\mathfrak{h}}^\e,\widetilde{\mathfrak{h}},\widetilde{\P} )$ coincides with the law of $(\mathfrak{h}^\e,\mathfrak{h},\mathscr{\P})$  and 
				$\widetilde{\mathfrak{h}^\e}$ converges weakly  to $\widetilde{\mathfrak{h}}$, as $\e\to 0$,  that is,
				$$ \langle \widetilde{\mathfrak{h}}^\e-\widetilde{\mathfrak{h}}, g\rangle_{\mathcal{H}}=0, \ \text{ for all } \ g\in \mathcal{H},\ \widetilde{\P}\text{-a.s}. $$
				Let $\widetilde{\mathrm{Z}}_{\e,\widetilde{\mathfrak{h}}^\e}$ and $\widetilde{\mathrm{Z}}_{\widetilde{\mathfrak{h}}}$ be the solutions to  equations similar to \eqref{CE1} and \eqref{CE2}, respectively, where $\mathfrak{h}^\e$ is replaced by $\widetilde{\mathfrak{h}}^\e, {\boldsymbol{\beta}}(\cdot)$ by $\widetilde{\boldsymbol{\beta}}(\cdot)$ and $\mathfrak{h}$ by $\widetilde{\mathfrak{h}}$ . In order to prove this Lemma, we only need to show that 
				\begin{align}\label{4p25}
					\lim_{\e\to0}\widetilde{\P}\bigg(\sup_{t\in[0,T]}\|\widetilde{\mathrm{Z}}_{\e,\widetilde{\mathfrak{h}}^\e}(t)-\widetilde{\mathrm{Z}}_{\widetilde{\mathfrak{h}}(t)}\|_{\L^p}\geq \varpi\bigg)=0.
				\end{align}
				Define a sequence of stopping times
				\begin{align}\label{435}
					\tau^{\rho,\e}:=\inf\big\{t\geq 0: \|\widetilde{\mathrm{Z}}_{\e,\mathfrak{h}^\e}(t)\|_{\L^p}\wedge \|\widetilde{\mathrm{Z}}_\mathfrak{h}(t)\|_{\L^p}> \rho\big\}\wedge T,
				\end{align}
				where $\rho$ is some constant large enough. 
				From the estimates \eqref{4p8.1} and \eqref{4p9}, we get 
				\begin{align}\label{STP}
					\lim_{\rho\rightarrow \infty}\sup_{\e\in (0,1]}\widetilde{\P}\big(\tau^{\rho,\e}< T\big)=0.
				\end{align} 
				The above argument can be justified in the following way: Using the definition of the stopping time given in \eqref{435}, we infer 
				\begin{align*}
					\sup_{\e\in(0,1]}	\widetilde{\P}\big(\tau^{\rho,\e}< T\big)\leq \sup_{\e\in(0,1]} \widetilde{\mathbb{P}}\bigg(\sup_{t\in[0,T]}\|\widetilde{\mathrm{Z}}_{\e,\mathfrak{h}^\e}(t)\|_{\L^p}>\rho\bigg)\to 0,\ \text{ as }\ \rho\to\infty,
				\end{align*}
				by using \eqref{4p8.1}.  In view of \eqref{4p25}, \eqref{STP} and Lemma \ref{lem4.1.1} (see \eqref{3.11} also), we need only to prove the following:
				\begin{align}\label{4p37}
					\lim_{\e\to0}\widetilde{\P}\bigg(\sup_{t\in [0,T\wedge\tau^{\rho,\e}]}\|\widetilde{\mathrm{Z}}_{\e,\widetilde{\mathfrak{h}}^\e}(t)-\widetilde{\mathrm{Z}}_{\widetilde{\mathfrak{h}}(t)}\|_{\L^p}\geq \varpi\bigg)=0,
				\end{align}for any $\rho>0$ large enough. The above fact can be justified using a calculation similar  to the performed in \eqref{3.11}.
					
					Now, with the help of \eqref{CE1} and \eqref{CE2}, we  have for all $t\in[0,T],$ $ \P$-a.s.
					\begin{align}\label{4.14}
						&\widetilde{\mathrm{Z}}_{\e,{\widetilde{\mathfrak{h}}}^\e}(t,x)-\widetilde{\mathrm{Z}}_{\widetilde{\mathfrak{h}}}(t,x)\nonumber \\
						&\nonumber
						=\frac{\alpha}{\delta+1}\int_0^t\int_0^1\frac{\partial G}{\partial y}(t-s,x,y)\\&\nonumber \qquad\times\bigg[\frac{1}{\sqrt{\e} \lambda(\e)}\big(\p(u_0(s,y)+\sqrt{\e} \lambda(\e)\widetilde{\mathrm{Z}}_{\e,{\widetilde{\mathfrak{h}}}^\e}(s,y))-\p(u_0(s,y))\big)
						-\p'(u_0(s,y))\widetilde{\mathrm{Z}}_{\widetilde{\mathfrak{h}}}(s,y)\bigg] \d y\d s\\
						&\nonumber\quad  +\beta\int_0^t\int_0^1G(t-s,x,y)\\&\nonumber \qquad\times\bigg[\frac{1}{\sqrt{\e} \lambda(\e)}\big(\c(u_0(s,y)+\sqrt{\e} \lambda(\e)	\widetilde{\mathrm{Z}}_{\e,{\widetilde{\mathfrak{h}}}^\e}(s,y))-\c(u_0(s,y))\big) -\c'(u_0(s,y))\widetilde{\mathrm{Z}}_{\widetilde{\mathfrak{h}}}(s,y)\bigg]\d y\d s\\
						&\nonumber\quad 
						+\sum_{j\in\N}\int_0^t\int_0^1G(t-s,x,y)\\&\nonumber \qquad\times\bigg[\frac{1}{ \lambda(\e)}\big(\sigma_j(s,y,(u_0+\sqrt{\e} \lambda(\e)	\widetilde{\mathrm{Z}}_{\e,{\widetilde{\mathfrak{h}}}^\e})(s,y))-\sigma_j(s,y,u_0(s,y))\big)\bigg]
						\d y\d \widetilde{\beta}_j(s)
						\\&\nonumber\quad + 
						\sum_{j\in\N}\int_0^t\int_0^1G(t-s,x,y)\\&\nonumber \qquad\times\bigg[\sigma_j(s,y,(u_0+\sqrt{\e} \lambda(\e)	\widetilde{\mathrm{Z}}_{\e,{\widetilde{\mathfrak{h}}}^\e})(s,y))\dot{{\widetilde{\mathfrak{h}}}}^\e_j(s) -\sigma_j(s,y,u_0(s,y))\dot{{\widetilde{\mathfrak{h}}}}_j(s)\bigg]\d y\d s
						\\&
						=:\frac{\alpha}{\delta+1}J_{1}^{\e}(t,x)+\beta J_{2}^{\e}(t,x)+J_{3}^{\e}(t,x)+J_{4}^{\e}(t,x).
					\end{align}
					For the sake of simplicity, we will no longer be using the tilde (~$\widetilde{\cdot}$~) notation.
					Let us consider the first term $J_1$ in the right-hand side of the expression \eqref{4.14} and split into the following manner:
					\begin{align}\label{4.15}
						&J_{1}^{\e}(t,x)
						\nonumber \\ &\nonumber=\int_0^t\int_0^1\frac{\partial G}{\partial y}(t-s,x,y)\bigg[\frac{1}{\sqrt{\e} \lambda(\e)}\big(\p(u_0(s,y)+\sqrt{\e} \lambda(\e)\mathrm{Z}_{\e,\mathfrak{h}^\e}(s,y))-\p(u_0(s,y))\big)
						\nonumber \\ &\nonumber\qquad \quad 
						-\p'(u_0(s,y))\mathrm{Z}_{\e,\mathfrak{h}^\e}(s,y)\bigg] \d y\d s \\&\nonumber\quad +\int_0^t\int_0^1\frac{\partial G}{\partial y}(t-s,x,y)\bigg[\p'(u_0(s,y))\big(\Z_{\e,\mathfrak{h}^\e}(s,y)-\Z_{\mathfrak{h}}(s,y)\big)\bigg]\d y \d s
						\\&=:J_{11}^{\e}(t,x)+J_{12}^{\e}(t,x).
					\end{align}
					By Taylor's formula, there exists an $\eta^{\e}_{1}$ taking values in $(0,1)$, such that
					\begin{align}
						\p(u_0+\sqrt{\e} \lambda(\e)\mathrm{Z}_{\e,\mathfrak{h}^\e})-\p(u_0)
						\nonumber =\sqrt{\e} \lambda(\e)\p'\big(u_0+\sqrt{\e} \lambda(\e)\eta^{\e}_{1}\mathrm{Z}_{\e,\mathfrak{h}^\e}\big)\mathrm{Z}_{\e,\mathfrak{h}^\e}.
					\end{align}
					Again, by Taylor's formula, there exists an $\eta^{\e}_{2}$ taking values in $(0,1)$, such that
					\begin{align}\label{p11}\nonumber
						&	\big|\p'\big(u_0+\sqrt{\e} \lambda(\e)\eta^{\e}_{1}\mathrm{Z}_{\e,\mathfrak{h}^\e}\big)\mathrm{Z}_{\e,\mathfrak{h}^\e}
						-\p'(u_0)\mathrm{Z}_{\e,\mathfrak{h}^\e}\big|
						\\&\nonumber\leq \sqrt{\e} \lambda(\e)\big|\p''\big(u_0+\sqrt{\e} \lambda(\e)\eta^{\e}_{2}\mathrm{Z}_{\e,\mathfrak{h}^\e}\big)\big|\big|\mathrm{Z}_{\e,\mathfrak{h}^\e}\big|\\&\leq C(\delta)|u_0|^{\delta-1}\big|\mathrm{Z}_{\e,\mathfrak{h}^\e}\big|^2+C(\delta)(\sqrt{\e}{\lambda(\e)})^{\delta-1}\big|\mathrm{Z}_{\e,\mathfrak{h}^\e}\big|^{\delta+1}.
					\end{align}
					Using the inequality \eqref{p11}, the estimate \eqref{A2}, H\"older's, Minkowski's and Young's inequality, we estimate the  term $\|J_{11}^{\e}(\cdot)\|_{\L^p}$ as 
					\begin{align}\label{J_11}\nonumber
						&\|J_{11}^{\e}(t)\|_{\L^p}
						\\&\nonumber\leq 	\e \lambda(\e)C(\delta)\int_{0}^{t}
						(t-s)^{-1+\frac{1}{2p}}\bigg(\|u_0(s)\|_{\L^{\frac{p(\delta-1)}{p-2}}}^{2(\delta-1)}+ \|\mathrm{Z}_{\e,\mathfrak{h}^\e}(s)\|_{\L^p}^{4}\bigg)\d s
						\\&\quad
						+C(\delta)\e^{\frac{\delta+1}{2}}{\lambda^{\delta+1}(\e)}
						\int_{0}^{t}
						(t-s)^{-1+\frac{1}{2p}}\bigg(\big\|\mathrm{Z}_{\e,\mathfrak{h}^\e}(s)\big\|_{\L^p}^{4}
						+\big\|\mathrm{Z}_{\e,\mathfrak{h}^\e}(s)\big\|_{\L^{\frac{p(\delta-1)}{p-2}}}^{2(\delta-1)}\bigg)\d s. 
					\end{align}
					Taking supremum from $0$ to $T\wedge\tau^{\rho,\e}$ and then taking expectation in \eqref{J_11}, we find
					{	\begin{align}\label{4.16}
							&\E\bigg[\sup_{t\in [0,T\wedge \tau^{\rho,\e}]}\|J_{11}^{\e}(t)\|_{\L^p}\bigg]
							\nonumber\\&\nonumber
							\leq\e \lambda(\e)C(\delta)\E\bigg[\int_{0}^{T\wedge\tau^{\rho,\e}}
							(t-s)^{-1+\frac{1}{2p}}\bigg( \sup_{r\in [0,s\wedge \tau^{\rho,\e}]}\|u_0(r)\|_{\L^{\frac{p(\delta-1)}{p-2}}}^{2(\delta-1)}+\sup_{r\in [0,s\wedge \tau^{\rho,\e}]}\big\|\mathrm{Z}_{\e,\mathfrak{h}^\e}(r)\big\|_{\L^p}^{4}\bigg)\d s\bigg]\\&\nonumber\quad+\e^{\frac{\delta+1}{2}}{\lambda^{\delta+1}(\e)}C(\delta)\E\bigg[
							\int_{0}^{t}
							(t-s)^{-1+\frac{1}{2p}}	\\&\qquad\quad\nonumber	\times \bigg(\sup_{r\in [0,s\wedge \tau^{\rho,\e}]}\big\|\mathrm{Z}_{\e,\mathfrak{h}^\e}(r)\big\|_{\L^p}^{4}
							+\sup_{r\in [0,s\wedge \tau^{\rho,\e}]}\big\|\mathrm{Z}_{\e,\mathfrak{h}^\e}(r)\big\|_{\L^{\frac{p(\delta-1)}{p-2}}}^{2(\delta-1)}\bigg)\d s\bigg]
							\\&
							\leq C(\delta, T, \rho)\left(\e \lambda(\e)+(\e{\lambda^{2}(\e)})^{\frac{\delta+1}{2}}\right).
					\end{align}}
					Using a similar calculation to \eqref{4.16}, we estimate the term $\E\bigg[\sup\limits_{t\in [0,T\wedge \tau^{\rho,\e}]}\|J_{12}^{\e}(t)\|_{\L^p}\bigg]$ as
					\begin{align}\label{4.17}\nonumber
						&	\E\bigg[\sup_{s\in [0,T\wedge \tau^{\rho,\e}]}\|J_{12}^{\e}(t)\|_{\L^p}\bigg]\\&\leq C(\alpha,\delta)\E\bigg[\int_{0}^{T\wedge \tau^{\rho,\e}}(t-s)^{-1+\frac{1}{2p}}
						\sup_{r\in [0,s\wedge \tau^{\rho,\e}]}\big\|\p'(u_{0}(r))\big(\mathrm{Z}_{\e,\mathfrak{h}^\e}(r)-\mathrm{Z}_\mathfrak{h}(r)\big)\big\|_{\L^1}\d s\bigg]
						\nonumber	\\ &\nonumber
						\leq C(\alpha,\delta)\E\bigg[\int_{0}^{T\wedge \tau^{\rho,\e}}(t-s)^{-1+\frac{1}{2p}}
						\sup_{r\in [0,s\wedge \tau^{\rho,\e}]}	\|\p'(u_{0}(r))\|_{\L^{\frac{p}{p-1}}}\|\mathrm{Z}_{\e,\mathfrak{h}^\e}(r)-\mathrm{Z}_\mathfrak{h}(r)\|_{\L^p}\d s\bigg]
						\\&\leq C(\alpha,\delta,\rho)\int_{0}^{T}(t-s)^{-1+\frac{1}{2p}}\E\bigg[\sup_{r\in [0,s\wedge \tau^{\rho,\e}]}\|\mathrm{Z}_{\e,\mathfrak{h}^\e}(r)-\mathrm{Z}_\mathfrak{h}(r)\|_{\L^p}\bigg]\d s.	
					\end{align}
					Combining \eqref{4.16} and \eqref{4.17} in \eqref{4.15}, we arrive at 
					\begin{align}\label{4.18}\nonumber
						&	\E\bigg[\sup_{t\in [0,T\wedge \tau^{\rho,\e}]}\|J_{1}^{\e}(t)\|_{\L^p}\bigg]\\&\nonumber\leq C(\delta, T, \rho)\left(\e \lambda(\e)+(\e{\lambda^{2}(\e)})^{\frac{\delta+1}{2}}\right)\\&\quad 
						+C(\alpha,\delta,\rho)\int_{0}^{T}(t-s)^{-1+\frac{1}{2p}}\E\bigg[\sup_{r\in [0,s\wedge \tau^{\rho,\e}]}\|\mathrm{Z}_{\e,\mathfrak{h}^\e}(r)-\mathrm{Z}_\mathfrak{h}(r)\|_{\L^p}\bigg]\d s.	
					\end{align}
					Using a similar calculation to \eqref{4.15}, we can split the second term  $J_2^\e$ appearing in the right hand side of \eqref{4.14} as
					\begin{equation}\label{J_2}
						J_{2}^{\e}(t,x)=:J_{21}^{\e}(t,x)+J_{21}^{\e}(t,x),\end{equation}
					where
					\begin{align*}
						J_{21}^{\e}(t,x)
						&=\int_0^t\int_0^1G(t-s,x,y)\bigg[\frac{1}{\sqrt{\e} \lambda(\e)}\big(\c(u_0(s,y)+\sqrt{\e} \lambda(\e)	\mathrm{Z}_{\e,\mathfrak{h}^\e}(s,y))-\c(u_0(s,y))\big)
						\\ &\nonumber\quad\quad-\c'(u_0(s,y))\mathrm{Z}_{\e,\mathfrak{h}^\e}(s,y)\bigg]d y\d s,
					\end{align*}
					and 
					\begin{align*}
						J_{22}^{\e}(t,x)=\int_0^t\int_0^1G(t-s,x,y)\c'(u_0(s,y))\big(\mathrm{Z}_{\e,\mathfrak{h}^\e}(s,y)-\mathrm{Z}_\mathfrak{h}(s,y)\big) \d y \d s.
					\end{align*}By Taylor's formula, there exists  a $\zeta_1^\e$ taking values in $(0,1)$ such that (cf. \eqref{CLT3})
					\begin{align*}
						&\big|	\big(\c(u_0+\sqrt{\e} \lambda(\e)	\mathrm{Z}_{\e,\mathfrak{h}^\e})-\c(u_0)\big)-\c'(u_0)\mathrm{Z}_{\e,\mathfrak{h}^\e}\big| \\& \nonumber\leq \sqrt{\e} \lambda(\e) \big|\c''\big(u_0+\sqrt{\e} \lambda(\e)\zeta_1^\e \mathrm{Z}_{\e,\mathfrak{h}^\e}(s,y)\big)\big|\big|\mathrm{Z}_{\e,\mathfrak{h}^\e}\big|^2
						\\&\nonumber\leq \sqrt{\e} \lambda(\e) \bigg[\delta(\delta+1)(1-\gamma)\big|u_0+\sqrt{\e}\lambda(\e)\mathrm{Z}_{\e,\mathfrak{h}^\e}\big|^{\delta-1}
						+2\delta(2\delta+1)\big|u_0+\sqrt{\e}\lambda(\e)\mathrm{Z}_{\e,\mathfrak{h}^\e}\big|^{2\delta-1}\bigg]\big|\mathrm{Z}_{\e,\mathfrak{h}^\e}\big|^2
						\\&\nonumber
						\leq C(\delta, \gamma)\bigg[\big\{(\sqrt{\e} \lambda(\e))^\delta\big|u_0\big|^{\delta-1}+(\sqrt{\e} \lambda(\e))^{2\delta}\big|u_0\big|^{2\delta-1}\big\}\big|\mathrm{Z}_{\e,\mathfrak{h}^\e}\big|^2	
						+(\sqrt{\e}\lambda(\e))^{\delta}\big|\mathrm{Z}_{\e,\mathfrak{h}^\e}\big|^{\delta+1}\\& \qquad+(\sqrt{\e}\lambda(\e))^{2\delta}\big|\mathrm{Z}_{\e,\mathfrak{h}^\e}\big|^{2\delta+1}
						\bigg].
					\end{align*}
					Now, we consider the term $\|J_{21}^{\e}(\cdot)\|_{\L^p}$ and estimate it using the above inequality and a similar calculation to \eqref{J_11}, as
					\begin{align}\label{27p40}
						&\|J_{21}^{\e}(t)\|_{\L^p}\nonumber\\
						& \nonumber\leq C(p)\int_0^t(t-s)^{-\frac{1}{2}+\frac{1}{2p}}\bigg\|\frac{1}{\sqrt{\e} \lambda(\e)}\big(\c(u_0(s)+\sqrt{\e} \lambda(\e)	\mathrm{Z}_{\e,\mathfrak{h}^\e}(s))-\c(u_0(s))\big)	\\ \nonumber &\qquad	 
						\quad -\c'(u_0(s))\mathrm{Z}_{\e,\mathfrak{h}^\e}(s)\bigg\|_{\L^1}\d y\d s
						\\ \nonumber &
						\leq C(\delta, \gamma)\int_0^t(t-s)^{-\frac{1}{2}+\frac{1}{2p}}\bigg\{\big\|\big((\sqrt{\e} \lambda(\e))^{\delta}|u_0(s)|^{\delta-1}+(\sqrt{\e} \lambda(\e))^{2\delta}|u_0(s)|^{2\delta-1}\big)|\mathrm{Z}_{\e,\mathfrak{h}^\e}(s)|^2\big\|_{\L^1}
						\\ \nonumber &\qquad \quad 
						+(\sqrt{\e}\lambda(\e))^{\delta}\|\mathrm{Z}_{\e,\mathfrak{h}^\e}(s)\|^{\delta+1}_{\L^{\delta+1}}+(\sqrt{\e}\lambda(\e))^{2\delta}\|\mathrm{Z}_{\e,\mathfrak{h}^\e}(s)\|^{2\delta+1}_{\L^{2\delta+1}}\bigg\}\d s \\ \nonumber &
						\leq C(\delta, \gamma)\int_0^t(t-s)^{-\frac{1}{2}+\frac{1}{2p}}\bigg\{\left((\sqrt{\e} \lambda(\e))^{\delta}\|u_0(s)\|^{\delta-1}_{\L^{\frac{p(\delta-1)}{p-2}}}+(\sqrt{\e} \lambda(\e))^{2\delta}\|u_0(s)\|^{2\delta-1}_{\L^{\frac{p(2\delta-1)}{p-2}}}\right)\|\mathrm{Z}_{\e,\mathfrak{h}^\e}(s)\|^{2}_{\L^p}
						\\  &
						\qquad\quad +(\sqrt{\e}\lambda(\e))^{\delta}\|\mathrm{Z}_{\e,\mathfrak{h}^\e}(s)\|^{\delta+1}_{\L^{\delta+1}}+(\sqrt{\e}\lambda(\e))^{2\delta}\|\mathrm{Z}_{\e,\mathfrak{h}^\e}(s)\|^{2\delta+1}_{\L^{2\delta+1}}\bigg\}\d s.
					\end{align}
					Taking supremum from $0$ to  $T\wedge\tau^{\rho,\e}$ and then taking expectation on both sides of the inequality \eqref{27p40}, we get
					\begin{align}\label{4.20}
						&\E\bigg[\sup_{t\in [0,T\wedge \tau^{\rho,\e}]}\|J_{21}^{\e}(t)\|_{\L^p}\bigg] \nonumber
						\\ \nonumber &\leq C(\delta, \gamma)\E\bigg[\int_0^{T\wedge\tau^{\rho,\e}}(t-s)^{-\frac{1}{2}+\frac{1}{2p}} \\&\qquad \quad \times\nonumber
						\sup_{r\in [0,s\wedge \tau^{\rho,\e}]}\bigg\{\left((\sqrt{\e} \lambda(\e))^{\delta}\|u_0(r)\|^{\delta-1}_{\L^{\frac{p(\delta-1)}{p-2}}}+(\sqrt{\e} \lambda(\e))^{2\delta}\|u_0(r)\|^{2\delta-1}_{\L^{\frac{p(2\delta-1)}{p-2}}}\right)\|\mathrm{Z}_{\e,\mathfrak{h}^\e}(r)\|^{2}_{\L^p}
						\\ \nonumber &\qquad +
						(\sqrt{\e}\lambda(\e))^{\delta}\|\mathrm{Z}_{\e,\mathfrak{h}^\e}(r)\|^{\delta+1}_{\L^{\delta+1}}
						+(\sqrt{\e}\lambda(\e))^{2\delta}\|\mathrm{Z}_{\e,\mathfrak{h}^\e}(r)\|^{2\delta+1}_{\L^{2\delta+1}}\bigg\}\d s\bigg]
						\\  &
						\leq  \left((\sqrt{\e}\lambda(\e))^{\delta}+(\sqrt{\e}\lambda(\e))^{2\delta}\right)C(\delta, \gamma,p,  \rho, T,u_0).
					\end{align}
					Let us consider the term $	\E\bigg[\sup\limits_{t\in [0,T\wedge \tau^{\rho,\e}]}\|J_{22}^{\e}(t)\|_{\L^p}\bigg]$, and estimate it using a similar calculation to \eqref{4.17}, as
					\begin{align}\label{4.25}\nonumber
						&	\E\bigg[\sup_{t\in [0,T\wedge \tau^{\rho,\e}]}\|J_{22}^{\e}(t)\|_{\L^p}\bigg]\\&\leq C(\delta,\gamma)\int_{0}^{T}(t-s)^{-\frac{1}{2}+\frac{1}{2p}}\E\bigg[\sup_{r\in [0,s\wedge \tau^{\rho,\e}]}\|\mathrm{Z}_{\e,\mathfrak{h}^\e}(r)-\mathrm{Z}_\mathfrak{h}(r)\|_{\L^p}\bigg]\d s.		\end{align}
					Substituting the estimates \eqref{4.20} and \eqref{4.25} in \eqref{J_2}, we find
					\begin{align*}
						\E\bigg[\sup_{t\in [0,T\wedge \tau^{\rho,\e}]}\|J_{2}^{\e}(t)\|_{\L^p}\bigg] & \leq \left((\sqrt{\e}\lambda(\e))^{\delta}+(\sqrt{\e}\lambda(\e))^{2\delta}\right)C(\delta, \gamma,p,  \rho, T,u_0)\\&\quad+C(\delta,\gamma)\int_{0}^{T}(t-s)^{-\frac{1}{2}+\frac{1}{2p}}\E\bigg[\sup_{r\in [0,s\wedge \tau^{\rho,\e}]}\|\mathrm{Z}_{\e,\mathfrak{h}^\e}(r)-\mathrm{Z}_\mathfrak{h}(r)\|_{\L^p}\bigg]\d s.	
					\end{align*}

					From  \cite[pp. 37]{KKM}, we know that the process $J_{3}^{\e}$ has  a modification in the space $\C([0,T]\times [0,1])$. Thus,  $\frac{1}{\lambda(\e)}J_{3}^{\e}\to 0$, $\P$-a.s., as $\e\to 0$.
					
					Now, we consider the final term $J_{4}^{\e}$ in the right hand side of \eqref{4.14} and split it in the following way (cf. \eqref{4.15}):
					\begin{align}\label{4.26}\nonumber
						&	J_{4}^{\e}(t,x)\\&=\sum_{j\in\N}\int_0^t\int_0^1G(t-s,x,y)\sigma_j\big(s,y,(u_0+\sqrt{\e} \lambda(\e)	\mathrm{Z}_{\e,\mathfrak{h}^\e})(s,y)\big)\big(\dot{\mathfrak{h}}^\e_j(s)-\dot{\mathfrak{h}}_j(s)\big)\d y\d s
						\nonumber\\ \nonumber &\quad 
						+\sum_{j\in\N}	\int_0^t\int_0^1G(t-s,x,y)\big[\sigma_j\big(s,y,(u_0+\sqrt{\e} \lambda(\e)	\mathrm{Z}_{\e,\mathfrak{h}^\e})(s,y)\big)-\sigma_j(s,y,u_0(s,y))\big]\dot{\mathfrak{h}}_j(s)\d y\d s
						\\  &
						=:J_{41}^{\e}(t,x)+J_{42}^{\e}(t,x).
					\end{align}
					By employing similar calculations as in \eqref{4.10}, we estimate the term $\|J_{41}^{\e}(\cdot)\|_{\L^p}$ as
					\begin{align}\label{4.28}\nonumber
						&\|J_{41}^{\e}(t)\|_{\L^p}	\\ &
						\leq C(p)T^{\frac{1}{2p}+\frac{1}{4}}\bigg(\sum_{j\in\N}\int_0^{T\wedge\tau^{\rho,\e}}\big\|\sigma_j(s,\cdot,(u_0(s)+\sqrt{\e} \lambda(\e)	\mathrm{Z}_{\e,\mathfrak{h}^\e})(s,\cdot))\big(\dot{\mathfrak{h}^\e_j}(s)-\dot{\mathfrak{h}}_j(s)\big)\big\|_{\L^2}^2\d s \bigg)^\frac{1}{2}.
					\end{align}
					Using \eqref{4.11}, we obtain
					\begin{align}
						\lim_{\e\to 0}\E \bigg[\sup_{t\in [0,T\wedge \tau^{\rho,\e}]}\|J_{41}^{\e}(t)\|_{\L^p}	 \bigg]=0.
					\end{align}
							Using Hypothesis \ref{hyp1} (Lipschitz condition of $g$), Fubini's theorem, H\"older's, Minkowski's inequalities, \eqref{A1}, \eqref{A7} and Young's inequality for convolutions, we estimate the term  $\|J_{41}^{\e}(\cdot)\|_{\L^p}$  as (cf. \cite[Eq. (3.12)]{KKM})
							\begin{align}\label{J_42}\nonumber
								&\|J_{42}^{\e}(t)\|_{\L^p}\\&\nonumber
								=\bigg\|\sum_{j\in\N}	\int_0^t\int_0^1G(t-s,\cdot,y)\big[g\big(s,y,(u_0(s)+\sqrt{\e} \lambda(\e)	\mathrm{Z}_{\e,\mathfrak{h}^\e})(s)\big)-g(s,y,u_0(s))\big]
								\nonumber\\& \nonumber\qquad \times q_j\varphi_j(y)\dot{\mathfrak{h}}_j(s)\d y\d s\bigg\|_{\L^p}\\
								&	\nonumber\leq C(L) \sqrt{\e}\lambda(\e)
								\sum_{j\in\N}\int_0^t(t-s)^{\frac{1}{2p}-\frac{1}{2}}\|\mathrm{Z}_{\e,\mathfrak{h}^\e}(s)\varphi_j\|_{\L^1}q_j |\dot{\mathfrak{h}}_j(s)|\d s
								\\&\nonumber \leq C(L) \sqrt{\e}\lambda(\e)
								\sum_{j\in\N}\int_0^t(t-s)^{\frac{1}{2p}-\frac{1}{2}}\|\mathrm{Z}_{\e,\mathfrak{h}^\e}(s)\|_{\L^2}\|\varphi_j\|_{\L^2}q_j |\dot{\mathfrak{h}}_j(s)|\d s	\nonumber
								\\&\nonumber \leq 
								C(L) \sqrt{\e}\lambda(\e)\bigg(\sum_{j\in\N}\int_0^t\dot{\mathfrak{h}}_j^2(s)\d s\bigg)^{\frac{1}{2}}\bigg(	\sum_{j\in\N}\int_0^t(t-s)^{\frac{1}{p}-1}\|\mathrm{Z}_{\e,\mathfrak{h}^\e}(s)\|_{\L^2}^2\|\varphi_j\|_{\L^2}^2q_j^2\d s\bigg)^{\frac{1}{2}}
								\\&\nonumber\leq C(L)\sqrt{\e}\lambda(\e)\bigg(\int_0^t\|\dot{\mathfrak{h}}(s)\|_{\ell^2}^2\d s\bigg)^{1/2}
								\bigg(\sum_{j\in\N}q_j^2\bigg)^{\frac{1}{2}}\left(\int_0^t(t-s)^{\frac{1}{p}-1}\|\mathrm{Z}_{\e,\mathfrak{h}^\e}(s)\|_{\L^2}^2\d s\right)^{1/2}\nonumber\\&
								\nonumber\leq C(L)\sqrt{\e}\lambda(\e)\|\mathfrak{h}\|_{\mathcal{H}}
								\left(\int_0^t(t-s)^{\frac{1}{p}-1}\|\mathrm{Z}_{\e,\mathfrak{h}^\e}(s)\|_{\L^2}^2\d s\right)^{1/2}\\&\leq \nonumber
								C(L)\sqrt{\e}\lambda(\e)\|\mathfrak{h}\|_{\mathcal{H}}
								\left\{\sup_{s\in[0,t]}\|\mathrm{Z}_{\e,\mathfrak{h}^\e}(s)\|_{\L^2}\bigg(\int_0^t(t-s)^{\frac{1}{p}-1}\d s\bigg)^{1/2}\right\}\\&\leq 
								C(L)\sqrt{\e}\lambda(\e)T^{\frac{1}{2p}}\|\mathfrak{h}\|_{\mathcal{H}}
								\sup_{s\in[0,t]}\|\mathrm{Z}_{\e,\mathfrak{h}^\e}(s)\|_{\L^2},
							\end{align}for $p>\max\{6,2\delta+1\}$, where we have used the fact that $\sum\limits_{j\in\N}q_j^2 <\infty$ and $\{\varphi_j\}_{j\in\N}$ is an orthonormal family.
							
							Taking  superemum from $0$ to $T\wedge \tau^{\rho,\e}$ and then taking the expectation, we find
							\begin{align*}
								\E \bigg[\sup_{t\in [0,T\wedge \tau^{\rho,\e}]}\|J_{42}^{\e}(t)\|_{\L^p}\bigg]
								&	
								\leq 	C(L,T)\sqrt{\e}\lambda(\e)\E\bigg[\|\mathfrak{h}\|_{\mathcal{H}}
								\sup_{s\in[0,T\wedge \tau^{\rho,\e}]}\|\mathrm{Z}_{\e,\mathfrak{h}^\e}(s)\|_{\L^2}\bigg]
								\\ \nonumber &	
								\leq \sqrt{\e}\lambda(\e) C(L, T, \rho),
							\end{align*}where we have used the fact that $\mathfrak{h}\in \mathscr{P}^2_N$ and Proposition \ref{prop4.6}.
							
							Substituting the estimates of $J_{41}^\e$ and $J_{42}^{\e}$ in \eqref{4.26}, we deduce 
							\begin{align}
								\lim_{\e\to 0}\E\bigg[\sup_{t\in [0,T\wedge \tau^{\rho,\e}]}\|J_{4}^{\e}(t)\|_{\L^p}\bigg]=0.
							\end{align}
							Combining all the estimates of $J_{1}^\e$ to $J_{4}^\e$ in \eqref{4.14} produce 
							\begin{align}
								\lim_{\e\to 0}\E\bigg[\sup_{t\in [0,T\wedge \tau^{\rho,\e}]}\|\mathrm{Z}_{\e,\mathfrak{h}^\e}(t)-\mathrm{Z}_\mathfrak{h}(t )\|_{\L^p}\bigg]=0,
							\end{align}and hence \eqref{4p37} follows, which completes the proof.
							%
						\end{proof}
						
						\begin{proof}[Proof of the Theorem \ref{MDPthm}]
							We have established the verification of Condition \ref{Cond2} (1) and (2) in Lemmas \ref{compactness} and \ref{WC}, respectively. Therefore, combining Theorem \ref{thrmLDP} and its equivalency with LDP implies the required result, that is, LDP holds for the family $\{\mathrm{Z}_\e\}_{\e>0}$ in the space $\C([0,T];\L^p([0,1]))$, for $p>\max\{6,2\delta+1\}$.
						\end{proof}

						\appendix
						\renewcommand{\thesection}{\Alph{section}}
						\numberwithin{equation}{section}
						\section{Some Useful Results}\label{SUR}
						In the Appendix, we recall some useful results which are repeatedly used in the paper. 
						
						\vspace{2mm}
						\noindent
						\textbf{Green's function estimates:} The following estimates have been frequently used in this work (cf. \cite{IG,IGDN,AKMTM3})
						\begin{align}\label{A1}
							|G(t,x,y)| &\leq Ct^{-\frac{1}{2}}e^{-\frac{|x-y|^2}{a_1t}},\\
							\label{A2}
							\bigg|\frac{\partial G}{\partial y}(t,x,y)\bigg| &\leq Ct^{-1}e^{-\frac{|x-y|^2}{a_2t}},\\
							\label{A3}
							\bigg|\frac{\partial G}{\partial t}(t,x,y)\bigg| &\leq Ct^{-\frac{3}{2}}e^{-\frac{|x-y|^2}{a_3t}},\\
							\label{A4}
							\bigg|\frac{\partial^2 G}{\partial y\partial t}(t,x,y)\bigg| &\leq Ct^{-2}e^{-\frac{|x-y|^2}{a_4t}},\\
							\label{A5}
							|G(t,x,z)-G(t,y,z)| &\leq C|x-y|^{\vartheta}t^{-\frac{\vartheta}{2}-\frac{1}{2}}\max\bigg\{e^{-\frac{|x-z|^2}{a_5t}},e^{-\frac{|y-z|^2}{a_5t}}\bigg\},\\
							\label{A6}
							\bigg|\frac{\partial G}{\partial z}(t,x,z)-\frac{\partial G}{\partial z}(t,y,z)\bigg| &\leq C|x-y|^{\vartheta}t^{-1-\frac{\vartheta}{2}}\max\bigg\{e^{-\frac{|x-y|^2}{a_6t}},e^{-\frac{|x-z|^2}{a_6t}}\bigg\},
						\end{align}for all $t\in(0,T], x,y,z\in[0,1], C,a_i$ for $i=1,\ldots,6$, are some positive constants and $\vartheta\in[0,1]$.
						The following estimate (see \cite{IG}) is also  used frequently. 
						\begin{align}\label{A7}
							\big\|e^{-\frac{|\cdot|^2}{a(t-s)}}\big\|_{\L^p} \leq C (t-s)^{\frac{1}{2p}},
						\end{align} for any positive constant $a$ and $p\geq 1$.
						\appendix
						
						\medskip\noindent
						\textbf{Acknowledgments:}  The first author would like to thank  the  Department of Atomic Energy (DAE), Government of India for assisting through NBHM post-doctoral fellowship (File No.: 0204/6/2022/R\&D-II/5635). 
						This work was completed while the second author was a PhD student at Department of Mathematics, IIT Roorkee.
						M. T. Mohan would  like to thank the Department of Science and Technology (DST) Science $\&$ Engineering Research Board (SERB), India for a MATRICS grant (MTR/2021/000066).

						%
						%
						%
						%
						%
						%
						%
						%

					\end{document}